\documentclass[preprint,12pt,authoryear]{elsarticle}

\usepackage{amssymb}
\usepackage{amsmath}
\usepackage{amsthm}

\usepackage{graphicx}
\usepackage{tikz}
\usetikzlibrary{trees}
\usetikzlibrary {shapes.geometric}
\usepackage{xcolor}
\usepackage{booktabs}
\usepackage{multicol}

\newtheorem{theorem}{Theorem}
\newtheorem{lemma}{Lemma}
\newtheorem{corollary}{Corollary}
\newtheorem{proposition}{Proposition}

\newtheorem{definition}{Definition}
\newtheorem{example}{Example}

\newtheorem{remark}{Remark}

\newtheorem{algorithm}{Algorithm}

\newcommand{\D}{\mathcal{D}}
\newcommand{\F}{\mathcal{F}}
\newcommand{\Z}{\mathbb{Z}}
\newcommand{\N}{\mathbb{N}}
\newcommand{\R}{\mathbb{R}}
\newcommand{\Q}{\mathbb{Q}}
\newcommand{\C}{\mathbb{C}}

\newcommand{\ab}{\tfrac{a}{b}}

\journal{Discrete Mathematics}

\begin{document}

\begin{frontmatter}



\title{Digit expansions in rational and algebraic basis} %% Article title


\author{Lucía Rossi\footnote{This author was supported by the Austrian Science Fund, project ESP8098724.}} %% Author name

\affiliation{organization={Technical University of Vienna},%Department and Organization
            addressline={Wiedner Hauptstr. 8–10}, 
            city={Vienna},
            postcode={1040}, 
            country={Austria}}

%% Abstract
\begin{abstract}
Consider $\alpha \in \Q(i)$ satisfying $|\alpha| >1$. Let $\D = \{0,1,\ldots,|a_0|-1\}$, where $a_0$ is the independent coefficient of the minimal primitive polynomial of $\alpha$. We introduce a way of expanding complex numbers in base $\alpha$ with digits in $\D$ that we call $\alpha$-expansions, which generalize rational base number systems introduced in~\citep{MR2448050} and are related to rational self-affine tiles introduced in~\citep{MR3391902}. We define an algorithm to obtain the expansions for certain Gaussian integers and show results on the language. We then extend the expansions to all $x \in \C$ (or $x \in \R$ when $\alpha = \ab \in \Q$, the rational case will be our starting point) and show that they are unique almost everywhere. We relate them to tilings of the complex plane. We characterize $\alpha$-expansions in terms of $p$-adic completions of $\Q(i)$ with respect to Gaussian primes.
\end{abstract}



%% Keywords

\end{frontmatter}



%\section{Example Section}
%\label{sec1}
%% Labels are used to cross-reference an item using \ref command.

%Section text. See Subsection \ref{subsec1}.

%% Use \subsection commands to start a subsection.
%\subsection{Example Subsection}
%\label{subsec1}

%Subsection text.

%% Use \subsubsection, \paragraph, \subparagraph commands to 
%% start 3rd, 4th and 5th level sections.
%% Refer following link for more details.
%% https://en.wikibooks.org/wiki/LaTeX/Document_Structure#Sectioning_commands

\section{Introduction.}\label{section:introduction}

The study of rational and algebraic digit systems can be motivated by the following question: can we define something analogous to the decimal system where the base is a rational or an algebraic number? A desirable property of a number system is that almost all numbers (in a measure-theoretic sense) can be expanded uniquely. 

In~\citep{Kempner} and~\citep{Renyi:57}, expansions with respect to non-integral real bases are considered. A straightforward way to expand a real number in a non-integer base is through the greedy algorithm, as is the case for $\beta$-expansion introduced in~\citep{Parry60} (see also \citep{MR2024754}). One famous complex numeration system has base $-1+i$ and digits $\{0,1\}$, introduced in~\citep{P65}, related to the Twin Dragon introduced in~\citep{Knu}. The golden ratio is another famous base that gives rise to binary expansions avoiding the digit sequence $11$. The examples above are all algebraic integers, that is, roots of monic integer polynomials. We want to consider algebraic bases that are not algebraic integers.

Akiyama, Frougny and Sakarovitch~\citep{MR2448050} studied representations of numbers in a rational base $\tfrac{a}{b}$ where $a > b \geq 2$ are coprime integers, using the digits $\{0,1,\ldots,a-1\}$. These representations are obtained through the so-called\textit{ modified division algorithm}, which produces less significant digits first. In particular, they obtain a unique expansion for non-negative integers. Morgenbesser, Steiner, and Thuswaldner showed patterns in these number systems in~\citep{MST:13} related to uniform distribution of digits. In the first section of this paper, we will recall rational base number systems, relate them to $ p$-adic numbers, and use them as a starting point to define algebraic number systems.

Digit systems are very often linked to self-affine sets. \textit{Rational self-affine tiles} were introduced in~\citep{MR3391902} for algebraic bases by defining a certain non-Euclidean representation space in terms of $\mathfrak{p}$-adic completions of algebraic number fields with respect to prime ideals. We revisit some results of this paper (which is of great depth and complexity) and reframe some of their proofs in our setting. The article \citep{RT:22} serves as an introduction to the topic, as it centers around the particular example of the rational number system with base $-\frac32$ and digits $\D = \{0,1,2\}$ using the representation space $\R \times \Q_2$; Donald Knuth called the elements of this space \textit{ambinumbers} in~\citep{Knuth22}, where he continued the study of the $-\frac32$ example. In~\citep{RST:23}, the authors generalize rational self-affine tiles in two directions: they consider basis that are expanding matrices with rational entries, and varying digit sets; the $p$-adic completions are replaced by projective limits.

The central idea of this paper is to consider digit expansions of complex numbers in base $\alpha \in \Q(i) \setminus \R$. We assume that $|\alpha| > 1$. The digit set $\D$ is of the form $\{0,1,\ldots,|a_0|-1\}$, where $a_0$ is the independent coefficient of the minimal primitive polynomial $P_\alpha$ of $\alpha$. We introduce a lattice $\Lambda_\alpha$ in the complex plane and the so-called integer $\alpha$-expansions, which we will afterwards extend to the whole complex plane and relate to tilings and $p$-adic completions to show uniqueness almost everywhere, as well as results regarding the language of the expansions.

 \subsection{Outline and main results}

\indent In Section \ref{Rational base number systems}, we revisit rational number systems introduced in~\citep{MR2448050}, we extend their algorithm to negative bases and characterize expansions in terms of $ p$-adic completions.

In Section~\ref{Algebraic number systems}, we introduce integer $\alpha$-expansions in base $\alpha \in \Q(i) \setminus \R$ for points in the lattice $\Lambda_\alpha := \Z[\alpha] \cap \alpha^{-1}\Z[\alpha^{-1}]$ through the backward division map $T_\alpha : \Lambda_\alpha \rightarrow \Lambda_\alpha, N \mapsto \frac{N-d}{\alpha}$ where $d \in \D$ is uniquely determined so that $\frac{N-d}{\alpha} \in \Lambda_\alpha$. 

The first issue we encounter is that this algorithm may not terminate, producing an expansion that is infinite to the left. In Section~\ref{Finiteness property and shift radix systems}, we show that this question is analogous to determining the finiteness property of certain shift radix systems, which is well studied, and provide an algorithm to determine this property for a given base~$\alpha$. 

An interesting aspect of algebraic number systems is the special shape of the expansions. In section~\ref{The language of alpha expansions}, we depict them as paths on trees and show results on the language $L_\alpha$ of these expansions. 

In Theorem~\ref{theorem:residueclasses} we show the following: let $L^k_\alpha$ be the language of all integer $\alpha$-expansions of length $k$. Then each word in $L^k_\alpha$ can be extended to the right to a word in $L^{k+1}_\alpha$ by all digits $d \in \D$ that belong to exactly one residue class modulo $a_2$, where $a_2$ is the leading coefficient of $P_\alpha$. Consequently, $\#L_\alpha^k \leq \left\lceil \frac{|\D|}{a_2} \right\rceil ^k$. Afterwards, we show that the length of the integer $\alpha$-expansion of $N$ behaves like $\log_{|\alpha|}(|N|)+ \mathcal{O}(1)$. We give algorithms for the addition and multiplication of integer $\alpha$-expansions.

The next natural step is to extend the expansions to all complex numbers; we call them $\alpha$-expansions. Their language corresponds to words in $D^\omega$ whose prefixes are in $L_\alpha$. We introduce them in Section~\ref{Expansion of complex numbers}, and show how to approximate the expansion of $x \in \C$ to any desired precision.

In Section~\ref{Tilings for alpha expansions}, we represent $\alpha$-expansions through a family of tiles with fractal boundaries and show in Proposition~\ref{proposition:tiling} that they yield a tiling of $\C$ (this is a consequence of \cite[Theorem 3]{MR3391902}). This will be used to show our main result, stated in Theorem~\ref{theorem:uniqueae}: that $\alpha$-expansions of complex numbers are unique almost everywhere with respect to the Lebesgue measure.

In Section \ref{p adic completions}, we consider $p$-adic completions of the field $\Q(i)$ where $p$ is a Gaussian prime. In Theorem~\ref{padiccompletions} we show the following: write $\alpha = \frac{num(\alpha)}{den(\alpha)}$ where $num(\alpha), den(\alpha) \in \Z[i]$ have no common Gaussian prime divisors. Then, an expansion
        \begin{equation}
             x = \sum_{j \leq k} d_j \alpha^j \qquad (d_j\in\mathcal{D})
        \end{equation}
    is an $\alpha$-expansion of $x$ if and only if it converges to $0$ in $K_p$ for every Gaussian prime $p$ dividing $den(\alpha)$.
 We consider pairs in a space that we denote $\C \times K_{den(\alpha)}$, and define $\alpha$-expansions of points in this space. We obtain uniqueness almost everywhere.

At the end of the paper, we state some generalizations and pose some open questions. We mention that the restriction that $\alpha \in \Q(i)$ is mostly done for simplicity and that one could consider a quadratic algebraic number instead; in that case, $\Q(\alpha)$ may not be a unique factorization domain, so the completions would be defined in terms of prime ideals.

\section{Rational base number systems}\label{Rational base number systems}
     
 Let $\ab \in\Q$ with $a$ and $b$ coprime integers and $a > b \geq 2$. Consider the digit set $\D = \{0, \ldots, a-1\}$. Akiyama, Frougny and Sakarovitch~\citep{MR2448050} introduced the following modified division algorithm. 
 
 Write $N= N_0$ and define $N_{j} \in \Z$ by
\begin{equation}\label{eq:algo}
bN_j = aN_{j+1} + d_j,
\end{equation}
where $d_j\in \mathcal{D}$ is the unique digit satisfying $d_j\equiv bN_j \mod{a}$. Induction yields 
\begin{equation}\label{eq:NNi}
N=\left( \ab \right)^{j+1}N_{j+1}+\tfrac{d_{j}}{b}\left( \ab \right)^{j}+\cdots+\tfrac{d_1}{b}\left( \ab \right)+ \tfrac{d_0}{b}.
\end{equation}
Since $N_{j+1} < N_j$, eventually $N_{k+1} = 0$ for some minimal $k$. Then
 \begin{equation}
    N = \frac{1}{b} \left( d_k\left( \ab \right)^k + d_{k-1}\left( \ab \right)^{k-1} 
+ \cdots + d_1\left( \ab \right) + d_0 \right)
 \qquad (d_j\in\mathcal{D})
\end{equation}
with $d_k \neq 0$. 

We consider an analogous expansion for negative bases and show that the algorithm ends.
 
 \begin{proposition}
Let $\ab \in\Q$, $\ab <0$ with $-a > b > 2$, and $\D = \{0, \ldots, |a|-1\}$. Let $N = N_0 \in \Z$. Then the sequence $(N_j)_{j \geq 0}$ obtained through \eqref{eq:algo} is eventually zero.
\end{proposition}

\begin{proof}

For every $j \geq 0$, we have
\begin{equation}\label{eq:recurrence}
     |N_{j+1}| = \Bigg| \frac{b}{a} N_j - \frac{d_j}{a} \Bigg| \leq \left| \frac{b}{a} \right| |N_j| + \left| \frac{d_j}{a} \right|.
\end{equation}

% Suppose $\ab > 0$. Since $d_j \geq 0$ and $\left| \frac{b}{a} \right| < 1$, it follows that
% \[
%     0 \leq N_{j+1} < N_j.
% \]
% Hence, the sequence $(N_j)$ is strictly decreasing and bounded below by $0$. Therefore, $N_{j+1}$ is eventually zero.

% Since $\ab < 0$, we have $a < 0$. From \eqref{eq:recurrence}, it follows that
% \[
%     |N_{j+1}| \leq \left| \frac{b}{a} \right| |N_j| + \left| \frac{d_j}{a} \right|.
% \]
Since $|\frac{b}{a}|<1$ and $d_j < |a|$, it follows that $ |N_{j+1}| < |N_j| + 1,$ thus $(|N_j|)_{j \geq 0}$ is bounded and non-increasing. Suppose that $|N_{j+1}| = |N_j|$ for some $j$. This implies that either $N_{j+1} = N_j$ or $N_{j+1} = -N_j$. 
If $N_{j+1} = N_j$, then 
\[
    b N_j = a N_j + d_j , \quad \mbox{so} \quad (b-a)N_j = d_j.
\]
Since $b-a = b + |a| > d_j$, we must have $N_j = d_j = 0$.

If $N_{j+1} = -N_j$, then
\[
    b N_j = -a N_j + d_j, \quad \mbox{so} \quad N_j = \frac{d_j}{a+b}.\]
Moreover, $d_j \equiv b N_j \mod{a}$. On the next step of the iteration, we have
\[
    b N_{j+1} = a N_{j+2} + d_{j+1}.
\]
Substituting $b N_{j+1} = -b N_j$, we find $d_{j+1} \equiv -b N_j \mod{a}.$ Since $d_j, d_{j+1} \in \{0, \ldots, |a|-1\}$ and $a < 0$, it follows that $d_{j+1} = -a - d_j.$ Substituting $b N_{j+1} = -b N_j = -b \frac{d_j}{a+b}$ into the recurrence, we have
\[
    -b \frac{d_j}{a+b} = a N_{j+2} - a - d_j,
\]
which simplifies to
\[
    N_{j+2} = 1 + \frac{d_j}{a+b} = 1 + N_j.
\]
Note that $N_j < 0$, so it follows that
\[
    |N_{j+2}| = |N_j + 1| < |N_j|.
\]

Thus $(N_j)_{j \geq 0}$ is eventually zero.
\end{proof}

We are going to slightly rephrase their algorithm to understand these expansions better. If we consider $N = N_0 \in b\Z $ and remove the preceding $\frac1b$ factor, then the expansions obtained are the same: define $N_j \in b\Z$ for $j\geq 0$ satisfying
\[
    N_j = \ab N_{j+1} + d_j
\]
 with $d_j \in \D$. In other words, we are considering the backward division map 
 \[T_{\ab}: b\Z \rightarrow b\Z, \quad N \mapsto \frac ba (N-d) \] 
 where $d \in \D$ is the unique digit such that $\frac ba (N-d) \in b\Z$. By restricting this map to the lattice $b\Z$ instead of the ring $\Z[\ab]$, we obtain a unique expansion\begin{equation}
    N =  d_k\left( \ab \right)^k + d_{k-1}\left( \ab \right)^{k-1} 
+ \cdots + d_1\left( \ab \right) + d_0 
 \in b\Z.
\end{equation} We called it the integer $\ab$-expansion of $N$.

The expansion can be extended to real numbers.

%\subsection{Expansion of real numbers}

    \begin{definition}
    Let $\ab \in\Q$ with $a$ and $b$ coprime integers, $|a| \geq b \geq 2$, and $\D = \{0, \ldots, a-1\}$. Let $x \in \R$. If $\frac ab > 0$, assume $x>0$. An $\ab$-expansion of $x$ is an expansion of the form

    \begin{equation}
        x = d_k\left( \ab \right)^k + \cdots + d_1\left( \ab \right) + d_0 + d_{-1}\left( \ab \right)^{-1} + \cdots \qquad (d_j\in\mathcal{D})
    \end{equation}
    
    with $d_k \neq 0$ such that $(d_k \ldots d_l)_{a/b}$ corresponds to the integer $\ab$-expansion of some $N \in b\Z$ for every $l \leq k$. We denote it
    \[
        x = (d_k \ldots d_0 . d_{-1}d_{-2}\ldots)_{a/b}.
    \]
    \end{definition}

    It follows from~\cite[Theorem 36]{MR2448050} that the $\ab$ expansion of reals is unique almost everywhere for positive $\ab$, and the negative case can be deduced along similar lines.

%\subsection{The b-adic numbers.}\label{section:badicnumbers}

    %  A complete residue set of $\Z[\ab] \mod \ab$ is given by $\D = \{0, \ldots, |a|-1\}.$ By taking the base $\ab$ and the digit set $\D$, we want to define an expansion for each real number.

% Every non-zero $y\in\mathbb{Q}_p$ can be written uniquely as a series 
% \begin{equation}\label{padicseries}
% y=\sum_{j=\nu_p(y)}^{\infty}c_jp^j\qquad ( c_j\in\{0,\ldots,p-1\}, \; c_{\nu_p(y)} \neq 0 ).
% \end{equation}
%  This series converges in $\mathbb{Q}_p$. Here, $\nu_p(y) $ extends the $p$-adic valuation to all of $\Q_p$, and naturally \eqref{eq:dp} holds in $\Q_p$. The ring of integers of $\Q_p$ is denoted by $\Z_p$ and called the ring of $p$-adic integers, which are exactly the points $y \in \Q_p$ such that $\nu_p(y) \geq 0$. 

Our next step is to characterize the $\ab$ expansions of reals in terms of $p$-adic completions, following~\citep{MR3391902}. For more on the topic of $p$-adic numbers, we refer the reader to~\citep{MR1865659}. We denote by $\nu_p$ the $p$-adic valuation on $\Q$ and by $\Q_p$ the corresponding $p$-adic field.

    \begin{theorem}
    Let $\ab \in\Q$ and $\D = \{0, \ldots, a-1\}$ as before. Let $x \in \R$. If $\frac ab > 0$, assume $x>0$. A series 
        \begin{equation}\label{rationalexpansion}
             d_k\left( \ab \right)^k  
            + \cdots + d_1\left( \ab \right) + d_0 + d_{-1}\left( \ab \right)^{-1} + \cdots  \qquad (d_j\in\mathcal{D})
        \end{equation}
    converging to $x$ in $\C$ is an $\ab$-expansion of $x$ if and only if it converges to $0$ in $\Q_p$ for every prime $p$ dividing $b$.
    \end{theorem}

    \begin{proof}
        Let $l \in \Z$ so that $l \leq k$. %By definition, $(d_k \ldots d_{l})_{a/b}$ corresponds to the integer $\ab$-expansion of some integer $N \in b\Z$.
        The sum of the first $k-l$ terms of the series \eqref{rationalexpansion} is
        \[
           S_l = d_k\left( \ab \right)^{k} + \cdots + d_{l+1}\left( \ab \right)^{l+1} + d_{l}\left( \ab \right)^{l}.
        \]
         Write 
        \[
            N_l = \left( \ab \right)^{-l} S_l = d_k\left( \ab \right)^{k-l} + \cdots + d_{l+1}\left( \ab \right) + d_{l}.
        \]
        
        Assume that \eqref{rationalexpansion} is an $\ab$-expansion of $x$, then $N_l = \left( \ab \right)^{-l} S_l \in b\Z$. Let $p$ be a prime dividing $b$, then $\nu_p(S_l) \geq -l+1$. The remaining tail of the series \eqref{rationalexpansion} is given by
        \[
            d_{l-1}\left( \ab \right)^{l-1} + d_{l-2}\left( \ab \right)^{l-2} + \cdots
        \]
        Note that $\nu_p( d_t \left( \ab \right)^t ) \geq -l+1$ for every $t \leq l-1$. Denote by $L_p \in \Q_p$ to the limit of~\eqref{rationalexpansion} in $\Q(p)$, then $\nu_p(L_p) \geq -l+1$. Hence, letting $l \to -\infty$, the series~\eqref{rationalexpansion} converges to $0$ in $\Q_p$.
        
        For the converse, suppose there exists some $l$ such that $N_l \notin b\Z$. Then, there exists some prime $p$ dividing $b$ and $r \geq 1$ such that $p^r$ divides $b$ and $p^r$ does not divide $N_l$. Then $\nu_p(S_l) < r(-l+1)$. It holds that $\nu_p( d_t \left( \ab \right)^t ) \geq r(-l+1)$ for every $t \leq l-1$, so $\nu_p(L_p)  < r(-l+1)$ and therefore $L_p \neq 0$, that is, the series does not converge to $0$ in~$\Q_p$.
    \end{proof}

    In the upcoming sections, we will generalize these expansions to the complex plane, where one can visualize them through tilings.
    
%  \subsection{Ambinumbers}

%  Given a real number $x$, there are in principle many ways to expand it in base $\ab$ and digits $\D = \{0,\ldots,|a|-1\}$. We chose a particular kind of expansion, which is the one of the form \eqref{rationalexpansion} that converges to $0$ in $\Q_b$. It seems therefore reasonable, given a pair $(x,y) \in \R \times \Q_b$, to look for a series \eqref{rationalexpansion} that converges to $x$ in $\R$ and to $y$ in $\Q_b$. We call the pair $(x,y)$ an ambinumber. This motivates the following definition.

%  \begin{definition}
%     Consider a base $\ab$ and digits $\D = \{0,\ldots,|a|-1\}$. Given an ambinumber $(x,y) \in \R \times \Q_b$, we call a series
%         \begin{equation}\label{rationalexpansion}
%             \pm \left( d_k\left( \ab \right)^k  
%             + \cdots + d_1\left( \ab \right) + d_0 + d_{-1}\left( \ab \right)^{-1} + \cdots \right) \qquad (d_j\in\mathcal{D})
%         \end{equation}
%     an $\ab$ expansion of $(x,y)$ if the series converges to $x$ in $\R$ and to $y$ in $\Q_b$. 
%  \end{definition}
    
% The $\ab$-expansion of $x \in \R$ is then the $\ab$-expansion of the ambinumber $(x,0)$. The integer $\ab$-expansion of $N \in b\Z$ is the $\ab$-expansion of the ambinumber $(N,N)$.

\section{Algebraic number systems}\label{Algebraic number systems}

Consider a Gaussian rational $\alpha \in \Q(i)$ with $|\alpha| > 1$ and $\alpha \notin \R$. Then $\alpha$ is a quadratic algebraic number, and is a root of$$P_{\alpha}(X) = a_2X^2 + a_1X + a_0$$ such that $a_0, a_1, a_2 \in \Z$ are coprime, and assume $a_2 > 0$. We call $P_\alpha$ the \textbf{minimal primitive polynomial} of $\alpha$. %If $a_2 = 1$ we say that $\alpha$ is an algebraic integer. We assume $a_2 \neq 1$. %Note that the norm of $\alpha$ is $\Re(\alpha)^2+\Im(\alpha)^2 = \frac{|a_0|}{a_2} > 1$. 

Consider the set $\D = \{0,\ldots,|a_0|-1\}$. We call the pair $(\alpha, \D)$ an \textbf{algebraic number system}. The set $\D$ is a complete set of residues of $\Z[\alpha]$ modulo $\alpha$; this is known in the literature as a standard digit set.

%  \begin{lemma}\label{standarddigit}
%     Let $\alpha \in \Q(i)$ with $|\alpha| > 1$ and primitive minimal polynomial $P_\alpha(X) = a_2X^2 + a_1X + a_0.$ Then the set $\D = \{0,1,\ldots, |a_0|-1\} $ is a complete set of residues of $\Z[\alpha] \mod \alpha,$ that is,  $(\alpha, \D)$ is a standard digit system.
%  \end{lemma}

% \begin{proof}
%     Let $x \in \Z[\alpha].$ Then we can write
%     \[
%         x = x_k \alpha^k + \dots + x_1 \alpha + x_0
%     \]
%     with $x_k, \ldots, x_1, x_0 \in \Z$ and $k \in \N$. Note that $a_0 \neq 0,$ otherwise $P_\alpha$ would be reducible. Then there exists a unique $d \in \D$ such that $x_0 = a_0  \widetilde{x}_0 + d$ for $\widetilde{x}_0 \in \Z.$ This yields
%     \begin{equation*}
%         \begin{split}
%             x & = x_k \alpha^k + \dots + x_1 \alpha + a_0  \widetilde{x}_0 + d \\
%               & = x_k \alpha^k + \dots + x_1 \alpha - (a_n \alpha^n  + \cdots + a_1 \alpha) \widetilde{x}_0 + d
%         \end{split}
%     \end{equation*}
%     so $x \in \alpha \Z[\alpha]+d$ and this choice is unique.
% \end{proof}

%Note that $\D$ is a complete residue set of $\Z[\alpha] \mod \alpha$.

An \textbf{expansion} of $x \in \C$ in base $\alpha$ with digits in $\D$ (also known as a representation) is an infinite sequence $(d_j)_{j \leq k}$ with $k \in \Z$ and $d_j \in \D$ such that
    \begin{equation}\label{representationofx}
        x = \sum_{j \leq k} d_j \alpha^j.
    \end{equation}
    When $d_j=0$ for all $j<0$, the expansion is said to be an \textbf{integer expansion}. When $d_j=0$ for all $j \geq 0$, it is said to be a \textbf{fractional expansion}. %When $d_j=0$ for all $j<l$ for some $l \leq 0$, the expansion is said to be \textbf{finite}. % and denoted $$x = (d_k \ldots d_0)_\alpha.$$

    Given a number system $(\alpha, \D)$ and a sequence $(d_j)_{j \leq k}$, we define the \textbf{evaluation map} $\pi_\alpha$ as
    \[
        \pi_\alpha ((d_j)_{j \leq k}) = \sum_{j \leq k} d_j \alpha^j.
    \]

 % \begin{definition}
 %    Define the \textit{backward division mapping} $T_\alpha$ as
 %        \begin{equation}\label{backwarddivision}
 %            T_{\alpha} : \Z[\alpha] \rightarrow \Z[\alpha], \qquad x \mapsto \frac{x-d}{\alpha},
 %        \end{equation}
 %    where $d$ is the unique element of $\D$ such that $T_{\alpha}(x) \in \Z[\alpha]$.
 % \end{definition}

% When $\alpha$ is not an algebraic integer, that is, when $a_2 \neq 1$, the set $\Z[\alpha]$ is not discrete, because  Therefore, any $x \in \C$ has infinitely many different representations, which is not desired as we pursue uniqueness almost everywhere
 We have seen that rational number systems arise by restricting the backward division map to a lattice. We will generalize this idea after giving some definitions and results.

\begin{definition}
Let $$\Lambda_\alpha := \Z[\alpha] \cap \alpha^{-1} \Z[\alpha^{-1}],$$
with the structure of a $\Z$-module of $\Q(i)$.
\end{definition}

We now give a $\Z$-basis for $\Lambda_\alpha.$

\begin{lemma}
    Let $P_{\alpha}(X) = a_2X^2 + a_1X + a_0$ be the minimal primitive polynomial of $\alpha$. The $\Z$-module $\Lambda_\alpha$ is a lattice in $\C$ that can be expressed as
    \[
    \Lambda_\alpha = a_2 \Z + (a_2 \alpha + a_1) \Z.
    \]
 
\end{lemma}

\begin{proof}
See \cite[Lemma 6.14]{MR3391902}.
\end{proof}
   The set $$ \mathcal{B}_\alpha = \{a_2, a_2\alpha + a_1\}$$ is known as the \textbf{Brunotte basis} of $\Lambda_\alpha$. This lattice will be key in the definition of our algorithm.

 \begin{lemma}\label{standarddigit}
    $\D = \{0,1,\ldots, |a_0|-1\} $ is a complete set of residues of $\Lambda_\alpha / \alpha \Lambda_\alpha.$ Consequently, the {backward division map} $T_\alpha$ in $\Lambda_\alpha$ given by
        \begin{equation}\label{backwarddivision}
            T_{\alpha} : \Lambda_\alpha \rightarrow \Lambda_\alpha, \qquad x \mapsto \frac{x-d}{\alpha},
        \end{equation}
    is well defined, where $d$ is the unique element of $\D$ such that $T_{\alpha}(x) \in \Lambda_\alpha$.
 \end{lemma}

\begin{proof}
    Let $x \in \Lambda_\alpha$ be arbitrary. Then 
    \[
        x = u\,a_2 + v(a_2 \alpha+a_1)
    \]
   for unique $u,v \in \Z$. There are unique $w \in \Z$ and $d \in \D$ such that $u \, a_2 + v \, a_1 = w\,a_0 + d$. This yields
    \begin{equation*}
        \begin{split}
            x & = v\,a_2 \, \alpha + w\,a_0 + d \\
              & = v\,a_2 \, \alpha + w \, (-a_2 \, \alpha^2 - a_1 \, \alpha) + d \\
              & = \alpha \, (v \, a_2 -w \, (a_2 \, \alpha + a_1)) + d
        \end{split}
    \end{equation*}
    so $T_\alpha(x) = v \, a_2 -w (a_2 \, \alpha + a_1) \in \Lambda_\alpha$.
\end{proof}
 
% The lattice $\Lambda_\alpha$ will play, for algebraic number systems, an analogous role to that of the lattice $b\Z$ for number systems in base $\ab$. 

The following algorithm will yield the central definition of this paper.

\begin{algorithm}

Consider an algebraic number system $(\alpha, \D)$. Let $N = N_0 \in \Lambda_\alpha $ and define $N_{j} \in \Lambda_\alpha$ for $j \geqslant0$ by
\begin{equation}\label{eq:algoalpha}
N_j = \alpha N_{j+1} + d_j
\end{equation}
with $d_j \in \mathcal{D}$ uniquely defined by Lemma \ref{standarddigit} because $N_{j+1} = T_\alpha(N_j)$.
\end{algorithm}
%of the following: recall that $\D$ is a complete residue set of $\Z[\alpha] \mod \alpha,$ and note that $num(\alpha) N_{j+1} = \alpha den(\alpha) N_{j+1} \in \alpha \Z[\alpha]$, therefore $d_j$ is the unique representative of $den(\alpha) N_j \mod \alpha$ in $\D$.
Iteration of \eqref{eq:algoalpha} yields

\begin{equation}
    N = \alpha^{j+1} N_{j+1} + d_{j} \alpha^{j} + \cdots + d_1\alpha + d_0.
\end{equation}

%It is clear that, if $ N_{k+1} = 0 $ for some (minimal) $k \geq 0$, the algorithm gives an integer representation of $N$. 

\begin{definition}%[Integer $\alpha$-expansion]
    Suppose that the sequence $(N_j)_{j\geq0}$ defined in \eqref{eq:algoalpha} satisfies $N_{k+1}=0$ for some minimal $k$. Then
\begin{equation}
    N =  d_k \alpha ^k + d_{k-1} \alpha ^{k-1} 
+ \cdots + d_1 \alpha  + d_0 
 \qquad (d_j\in\mathcal{D}),
\end{equation}
with $d_k \neq 0$. We call this an \textbf{integer $\alpha$-expansion} of $N$, denoted 
$$N = (d_k\ldots d_0)_\alpha.$$
\end{definition}
We note
    $$\langle N \rangle_\alpha = d_k \ldots d_1 d_0.$$ The expansion of zero is the empty word $\varepsilon$. We proceed to show a couple of examples.
    
\begin{example}\label{example1}
     Let $\alpha = \frac{-1+3i}{2}$, root of $P_\alpha(X) = 2X^2 + 2X + 5$, and $\D = \{0,1,2,3,4\}$. %We can express $\alpha$ as a quotient of coprime Gaussian integers as
%\[  \alpha = \frac{-2+i}{1+i} .\]
The Brunotte basis for $\Lambda_\alpha$ is $ \mathcal{B}_\alpha = \{ 2,1+ 3i\}$.

Let $N = 1+3i \in \Lambda_\alpha$. Then
\[1+3i = \frac{-1+3i}{2}\,2 + 2 \]
so $N_1 = 2\in \Lambda_\alpha$ and $d_0 = 2$. The next iteration is
\[
    2 = \frac{-1+3i}{2}\,0 + 2
\]
so $d_1 = 2$ and $N_2 = 0$, and since we reached $0$ the algorithm stops. This yields the expansion $1+3i = (22)_\alpha$. More examples are shown in Table~\ref{example1table}.
 \end{example}

\begin{table}
    \centering
    \begin{tabular}{cc}
    \toprule
    $N$ & $\langle N \rangle_\alpha$ \\
    \cmidrule(lr){1-2}
    $-2 -6i$ & $2431$ \\
    $-3 -3i$ & $201$ \\
    $-4$ & $221$ \\
    $-5 +3i$ & $241$ \\
    $-6 +6i$ & $223011$ \\
    $-6i$ & $2433$ \\
    $-1 -3i$ & $203$ \\
    $-2$ & $223$ \\
    $-3 +3i$ & $243$ \\
    $-4 +6i$ & $223013$ \\
    $2 -6i$ & $2210$ \\
    $1 -3i$ & $2230$ \\
    \bottomrule
    \end{tabular}
    $\;\;\;$
    \begin{tabular}{cc}
    \toprule
    $N$ & $\langle N \rangle_\alpha$ \\
    \cmidrule(lr){1-2}
    $-1 +3i$ & $20$ \\
    $-2 +6i$ & $40$ \\
    $4 -6i$ & $2212$ \\
    $3 -3i$ & $2232$ \\
    $2$ & $2$ \\
    $1 +3i$ & $22$ \\
    $6i$ & $42$ \\
    $6 -6i$ & $2214$ \\
    $5 -3i$ & $2234$ \\
    $4$ & $4$ \\
    $3 +3i$ & $24$ \\
    $2 +6i$ & $44$ \\
    \bottomrule
    \end{tabular}
    
    \,\,
    \caption{Integer $\alpha$-expansion for $\alpha = \frac{-1+3i}{2}$ and $N =2\lambda + (1+3i)\mu \in \Lambda_\alpha$ for $\lambda, \mu \in \{ -2, \ldots, 2\}.$}
    \label{example1table}
\end{table}

\begin{example}
Let $\alpha = \frac{3+2i}{3}$, root of $P_\alpha(X) = 9x^2 - 2x + 13$, and $\D = \{ 0, \ldots, 12\}$. The lattice $\Lambda_\alpha$ is spanned by $\mathcal{B}_\alpha =\{9, -9+6i \}$. 

Let $6i \in \Lambda_\alpha$. We have 
\[
    6i = \frac{3+2i}{3} \, 6i +4,
\]
and therefore $N_j = 6i$ and $d_j = 4$ for all $j \geq 0$, and we do not obtain an integer $\alpha$-expansion (instead, we would get infinitely many $4$'s to the left). 
\end{example}

In the next section, we will show how to determine whether the algorithm always stops for a given base.

\section{Finiteness property and shift radix systems}\label{Finiteness property and shift radix systems}

Consider an algebraic number system $(\alpha, \D)$. We are only interested in bases $\alpha$ where there exists an integer $\alpha$-expansion of $N$ for every $N \in \Lambda_\alpha,$ and we want to determine when this happens. We now tackle this issue and relate it to a well-studied problem.

\begin{definition}%[Finiteness property]
    We say that $\alpha$ has the \textbf{finiteness property} in $\Lambda_\alpha$ if, for every $N \in  \Lambda_\alpha$, the sequence $(N_j)_{j \geq 0}$ defined by \eqref{eq:algoalpha} satisfies $N_j = 0$ for some $j \geq 0$.
\end{definition}

The previous example shows that $\alpha = \frac{3+2i}{3}$ does not have the finiteness property in~$\Lambda_\alpha$.

\begin{proposition}
    Suppose that $\alpha$ has the finiteness property. Then every $N \in \Lambda_\alpha$ has a unique integer $\alpha$-expansion.
\end{proposition}

\begin{proof}
    This follows directly from the definition: if $\alpha$ has the finiteness property, then an integer $\alpha$-expansion of $N$ exists for every $N \in \Lambda_\alpha$, and the digits $d_j$ are chosen uniquely through the formula \eqref{eq:algoalpha}.
\end{proof}

There is no simple characterization for the bases $\alpha$ where the finiteness property holds in $\Lambda_\alpha$. This question is analogous to determining the finiteness property of a family of dynamical systems known as shift radix systems. We will show that the backward division map $T_\alpha$ restricted to $\Lambda_\alpha$ is conjugate to an almost linear map defined in $\Z^2$.

Given the Brunotte basis $\mathcal{B}_\alpha = \{a_2, a_2\alpha + a_1\}$ of $\Lambda_\alpha$, consider the map
\begin{equation}\label{eq:iota}
    \iota_\alpha : \Q^2 \rightarrow \Q(i), \qquad (z_0, z_{1}) \mapsto \mbox{sgn}(a_0) (z_0 a_2 + z_1( a_2\alpha + a_1)).
\end{equation}
Then $z \in \Z^n$ if and only if $\iota_\alpha(z) \in \Lambda_\alpha.$ 

Given $r \in \R^n$, a \textbf{shift radix system} is a map of the form 
\[ 
    \tau_r : \Z^n \rightarrow \Z^n, \qquad (z_0, \ldots, z_{n-1}) \mapsto (z_1, \ldots, z_{n-1}, -\lfloor r \cdot z \rfloor), 
\] 
where ``$\cdot$'' denotes the scalar product in $\R^n.$ Consider the vector $r = (\frac{a_2}{a_0}, \frac{a_1}{a_0})$. The following map $\tau_r$ is a type of shift radix system (see~\citep{MR2735753} and~\citep{Kirschenhofer-Thuswaldner:14}):
\[ 
    \tau_r : \Z^2 \rightarrow \Z^2, \qquad (z_0, z_{1}) \mapsto \left(z_1,-\left\lfloor \tfrac{a_2}{a_0} z_0 + \tfrac{a_1}{a_0} z_1 
    \right\rfloor\right), 
\]

It can be seen by direct computation that, for any $(z_0, z_1) \in \Z^2$ and $r = (\frac{a_2}{a_0}, \frac{a_1}{a_0})$, 
\begin{equation}\label{eq:conjugacy}
    T_\alpha(\iota_\alpha(z_0, z_1)) = \iota_\alpha(\tau_r(z_0, z_1)).
\end{equation}

The map $\tau_r$ is said to have the finiteness property if, for every $(z_0, z_1) \in \Z^2$, there exists $k \in \N$ such that $\tau_r^k(z_0, z_1) = 0$. Since $\Z^2$ gets mapped to the lattice $\Lambda_\alpha$ via $\iota_\alpha$, $\alpha$ has the finiteness property in $\Lambda_\alpha$ if and only if $\tau_r$ has the finiteness property.

There is no easy way to describe the set of parameters $r$ such that $\tau_r$ has the finiteness property (this is a highly nontrivial problem that has been well studied), but there is a fairly simple algorithm to determine it using what is known as a set of witnesses. 

\begin{definition}
A \textbf{set of witnesses} associated with a parameter $r \in \R^n$ is a set $\mathcal{V}_r$ such that: 
\begin{enumerate}
    \item $\{ \pm e_1, \ldots, \pm e_n\} \subset \mathcal{V}_r$ (here $e_j$ is the $j$-th canonical vector) and 
    \item $z \in \mathcal{V}_r$ implies $\{ \tau_r(z), -\tau_r(-z) \} \subset \mathcal{V}_r$.
\end{enumerate}
\end{definition}

The next result can be found in~\cite[Theorem 5.1]{ABBPT:05}.

\begin{lemma}\label{lemma:witnesses}
    Given a set of witnesses $\mathcal{V}_r$ for a parameter $r$, the map $\tau_r$ has the finiteness property if and only if for every $z \in \mathcal{V}_r$ there is $k \in \N$ such that $\tau_r(z)^k = 0.$
\end{lemma}

In~\cite[Section 5]{ABBPT:05}, the authors give an algorithm to find a set of witnesses. We adapt it to our setting.

\begin{algorithm}\label{algorithm:witnesses}
    Let $(\alpha, \D)$ be an algebraic number system. Let $r = (\frac{a_2}{a_0}, \frac{a_1}{a_0})$ and $\tau_r$ as before. 
    \begin{enumerate}
        \item Let $\mathcal{V}_0 = \{ \pm (1,0), \pm (0,1) \}$.
        \item For $j \geq 0$, let $\mathcal{V}_{j+1} = \mathcal{V}_j \cup \tau_r(\mathcal{V}_j) \cup (-\tau_r(-\mathcal{V}_j))$.
        \item Since $\alpha$ is expanding, $r$ is contracting, so for some $j$ we reach $\mathcal{V}_{j+1} = \mathcal{V}_j = : \mathcal{V}_r$.
    \end{enumerate}
    Then $\mathcal{V}_r$ is a set of witnesses for $r = (\frac{a_2}{a_0}, \frac{a_1}{a_0})$. Consider the map $\iota_\alpha$ as before and define $\mathcal{V}_\alpha := \iota_\alpha(\mathcal{V}_r).$ We call $\mathcal{V}_\alpha$ a set of witnesses for $(\alpha, \D)$.
\end{algorithm}

The following is a direct consequence of Lemma~\ref{lemma:witnesses} and equation~\eqref{eq:conjugacy}.

\begin{proposition}
    Let $\mathcal{V}_\alpha$ be a set of witnesses for $(\alpha, \D)$ as in Algorithm~\ref{algorithm:witnesses}. Then $\alpha$ has the finiteness property in $\Lambda_\alpha$ if and only if $N$ has an integer $\alpha$-expansion for every $N \in \mathcal{V}_\alpha$.
\end{proposition}

We conclude this section with an example.

\begin{example}
    We continue with $\alpha = \frac{-1+3i}{2}$ and $\D = \{0,1,2,3,4\}$. Let $r =  (\frac{a_2}{a_0}, \frac{a_1}{a_0}) = (\frac25, \frac25)$. Then  $\tau_r(z_0, z_{1}) = \left(z_1,-\left\lfloor \frac25 z_0 +\frac25 z_1 
    \right\rfloor\right).$ A set of witnesses is given by $$\mathcal{V}_r = \{ (0, 1), (-1, 1), (0, 0), (-1, 0), (0, -1), (1, 0), (1, -1) \}.$$ 
    Then $$\mathcal{V}_\alpha = \{ 1+3i, -1 + 3i, 0, -2, -1-3i, 2, 1-3i \}.$$ 
    Each point has an integer $\alpha$-expansion, as was shown in Table~\ref{example1table}.
    % \begin{equation*}
    % \begin{split}
    % 1+3i &=  (22)_\alpha \\
    % -1 + 3i &=  (20)_\alpha \\
    % 0 &= (0)_\alpha \\
    % -2 &=  (223)_\alpha \\
    % -1-3i &=  (203)_\alpha \\
    % 2 &= (2)_\alpha \\
    % 1 - 3i &= (2230)_\alpha .\\
    % \end{split}
    % \end{equation*}
     Therefore, $\alpha$ has the finiteness property in $\Lambda_\alpha$.
\end{example}

\section{The language of integer $\alpha$-expansions}\label{The language of alpha expansions}

Let $(\alpha,\D)$ be an algebraic number system. Denote by $\D^*$ the set of all \textbf{words} or finite sequences $d_k \ldots d_0$ with $d_j \in \D$ and $k \geq 0$, together with the empty word. This set becomes a free monoid with the operation of concatenation. Denote by $\D^{\omega}$ the set of infinite words with digits in $\D$.

% The monoid $\D^*$ is classically represented through the infinite full $|a_0|$-ary tree: every node has $|a_0|$ children with labels from $0$ to $|a_0|-1$. A word in $\D^*$ can be obtained starting from the root of the tree, going on a finite path, and writing down the label of each edge contained in the path. We will define an appropriate subtree to represent the integer $\alpha$-expansions, and we will later consider infinite paths on this tree to define expansions for all complex numbers.

\begin{definition}
     Consider a number system $(\alpha, \D)$ with the finiteness property. We define the \textbf{language} $L_\alpha$ as the subset of $\D^*$ given by the set of all integer $\alpha$-expansions. For simplicity, we allow words to start with $0$. We denote by $L_\alpha^k$ the words of length $k \geq 0$. %We denote by $L_\alpha^{\omega}$ to the set of all infinite words in $\D^*$ such that all of their prefixes are in $L_\alpha$. 
\end{definition}

The monoid $\D^*$ is classically represented through the infinite full $|a_0|$-ary tree. The language $L_\alpha$ is prefix closed, that is, given a word  $d_k \ldots d_0 \in L_\alpha,$ any prefix $d_k \ldots d_l, (0 \leq l \leq k)$ is also in $L_\alpha$ (this follows from the definition of integer $\alpha$-expansions). Consequently, $L_\alpha$ can be represented through a subtree $\mathcal{T}_\alpha$ of the full $|a_0|$-ary tree so that each path starting from the root corresponds to an integer $\alpha$-expansion (maybe preceded by zeroes).

More formally, for $N \in  \Lambda_\alpha$ consider
\[
    D(N) := \{ d \in \D \, : \, \alpha N + d \in  \Lambda_\alpha \}.
\]
The tree $\mathcal{T}_\alpha$ is constructed as follows: start by labeling the root $0$. The children of the root are labeled $d$ for every $d \in D(0).$ The node at the end of each child corresponds to $N = \pi_\alpha(d)$. Iteratively, each node corresponds to $\pi_\alpha(d_k \ldots d_0)$ where $d_k \ldots d_0$ is the label of the path starting from the root ending in the node, and the children of the node are labeled $D(\pi_\alpha(d_k \ldots d_0))$.

Figure~\ref{exampletree} depicts the first three levels of the tree $\mathcal{T}_\alpha$ for $\alpha = \frac{-1+3i}{2}$, with labeled edges. A first recognizable pattern is that the children of each node are either the odd or the even digits of $\D$.

\begin{figure}
\begin{tikzpicture} %[level distance=1.5cm, sibling distance=2.5cm]
        [level distance=2cm,
   level 1/.style={sibling distance=4.6cm, level distance = 2.5cm},
   level 2/.style={sibling distance=1.5cm, level distance = 1.4cm},
   level 3/.style={sibling distance=0.6cm, level distance = 1cm}]
    % Define styles for the nodes
    \tikzstyle{filled} = [circle, fill=black, minimum size=4pt, inner sep=0pt]
    
    % Root node
    \node[filled] {}
        child {node[filled] {}
                child {node[filled] {} 
                        child {node[filled] {} edge from parent node[left] {\small{0}}}
                        child {node[filled] {} edge from parent node[left, near end] {\small{2}}}
                        child {node[filled] {} edge from parent node[right] {\small{4}}}
                    edge from parent node[left] {\small{0}}}
                child {node[filled] {}
                        child {node[filled] {} edge from parent node[left] {\small{0}}}
                        child {node[filled] {} edge from parent node[left, near end] {\small{2}}}
                        child {node[filled] {} edge from parent node[right] {\small{4}}}
                    edge from parent node[left, near end] {\small{2}}}
                child {node[filled] {} 
                        child {node[filled] {} edge from parent node[left] {\small{0}}}
                        child {node[filled] {} edge from parent node[left, near end] {\small{2}}}
                        child {node[filled] {} edge from parent node[right] {\small{4}}}
                    edge from parent node[right] {\small{4}}}
            edge from parent node[left] {0\,\,}}
        child {node[filled] {}             
                child {node[filled] {} 
                        child {node[filled] {} edge from parent node[left] {\small{1}}}
                        child {node[filled] {} edge from parent node[right] {\small{3}}}
                    edge from parent node[left] {\small{0}}}
                child {node[filled] {} 
                        child {node[filled] {} edge from parent node[left] {\small{1}}}
                        child {node[filled] {} edge from parent node[right] {\small{3}}}
                    edge from parent node[left, near end] {\small{2}}}
                child {node[filled] {} 
                        child {node[filled] {} edge from parent node[left] {\small{1}}}
                        child {node[filled] {} edge from parent node[right] {\small{3}}}
                    edge from parent node[right] {\small{4}}}
            edge from parent node[left] {2}}
        child {node[filled] {} 
                child {node[filled] {} 
                        child {node[filled] {} edge from parent node[left] {\small{0}}}
                        child {node[filled] {} edge from parent node[left, near end] {\small{2}}}
                        child {node[filled] {} edge from parent node[right] {\small{4}}}
                    edge from parent node[left] {\small{0}}}
                child {node[filled] {} 
                        child {node[filled] {} edge from parent node[left] {\small{0}}}
                        child {node[filled] {} edge from parent node[left, near end] {\small{2}}}
                        child {node[filled] {} edge from parent node[right] {\small{4}}}
                    edge from parent node[left, near end] {\small{2}}}
                child {node[filled] {} 
                        child {node[filled] {} edge from parent node[left] {\small{0}}}
                        child {node[filled] {} edge from parent node[left, near end] {\small{2}}}
                        child {node[filled] {} edge from parent node[right] {\small{4}}}
                    edge from parent node[right] {\small{4}}}
            edge from parent node[right] {\,\,4}};
\end{tikzpicture}
\caption{The tree $\mathcal{T}_\alpha$ for $\alpha = \frac{-1+3i}{2}$.}
\label{exampletree} 
\end{figure}
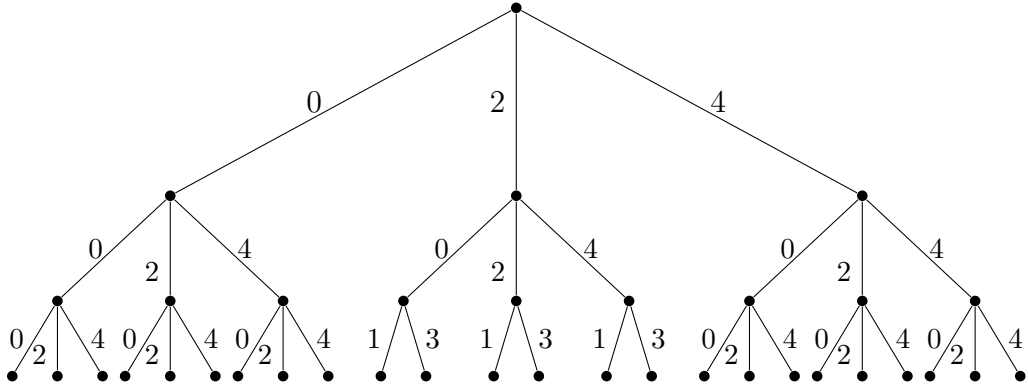

% We can start to understand the language $L_\alpha$ with some elementary arithmetic. We have $w_0 \in \Lambda_\alpha =  a_2 \Z + (a_2 \alpha + a_1) \Z$, therefore $w_2 \equiv 0 \mod a_2$, which shows that the first digit can be any $d \in \D$ in the residue class of $0$ modulo $a_2$.

% The next digit $w_1$ is chosen so that $w_0 \alpha + w_1 \in \Lambda_\alpha$ (because it corresponds to an $\alpha$-expansion), so 
% \[
%     \frac{w_0}{a_2} (a_2 \alpha + a_1) +  w_1 - \frac{w_0}{a_2} a_1 \in a_2 \Z +  (a_2 \alpha + a_1) \Z
% \]
% and consequently $w_1 \equiv \frac{w_0}{a_2} a_1 \mod a_2$.

% Write $w_1 = \frac{w_0}{a_2} a_1  + k a_2$ (then $k = \frac{w_1}{a_2} - \frac{w_0 a_1}{a_2^2} \in \Z)$. To obtain the next digit, we need
% \begin{equation}
%     \begin{split}
%         w_0 \alpha^2 + w_1 \alpha + w_2 & \in a_2 \Z + (a_2 \alpha + a_1) \Z \\
%         = \frac{w_0}{a_2} a_2 \alpha^2 + \frac{w_0}{a_2} a_1 \alpha + k a_2 \alpha + w_2 & \in a_2 \Z + (a_2 \alpha + a_1) \Z. \\        
%     \end{split}
% \end{equation}
% Using that $a_2 \alpha^2 + a_1 \alpha = -a_0,$ we get
% \[
%     -\frac{a_0}{w_0}a_2 + ka_2 \alpha + w_2 \in a_2 \Z + (a_2 \alpha + a_1) \Z,
% \]
% and therefore
% \[
%  w_2 \equiv \frac{a_0}{w_0}a_2 + k a_1 \mod a_2 \equiv \frac{a_0}{w_0}a_2 + \left( \frac{w_1}{a_2} - \frac{w_0 a_1}{a_2^2} \right) a_1 \mod a_2.
% \]

The following theorem establishes that each node will branch into exactly one residue class modulo $a_2$.

\begin{theorem}\label{theorem:residueclasses}
    A word in $L_\alpha$ can be extended to the right by all digits $d \in \D$ that belong to exactly one residue class modulo $a_2$. In other words, for every $N \in \Lambda_\alpha$ there exists some $r \in \Z$ such that
$$ D(N) = \D \cap (a_2 \Z + r).$$

\end{theorem}

\begin{proof}
    If $N \in \Lambda_\alpha$ then there exist unique $\lambda, \mu \in \Z$ such that $N = \lambda a_2 + \mu (a_2 \alpha + a_1)$. Using that $a_2 \alpha^2 + a_1 \alpha = -a_0,$ it holds that
    \begin{equation}
        \begin{split}
            \alpha N + d &= \lambda a_2 \alpha + \mu (a_2\alpha^2 + a_1 \alpha) \\
            &= \lambda (a_2 \alpha + a_1) - \lambda a_1 - \mu a_0 +d,
        \end{split}
    \end{equation}
    so $\alpha N + d \in \Lambda_\alpha$ if and only if $- \lambda a_1 - \mu a_0 +d \in a_2\Z$, that is, if $d \equiv \lambda a_1 + \mu a_0 \mod a_2$.
\end{proof}

The following corollary is immediate.

\begin{corollary}\label{corollary:integerexpansions}
    Let $L_\alpha^k$ be the set of $\alpha$-expansions of length~$k$. Then the cardinality of $L_\alpha^k$ satisfies
    \[ \#L_\alpha^k \leq \left\lceil \frac{|\D|}{a_2} \right\rceil ^k.  \]
\end{corollary}

We show in Figure~\ref{exampletree2} the tree for $\alpha = \frac{-1+5i}{3}$, root of $9X^2 + 6X + 26$. Note that each node has exactly three children except one whose children are congruent with $8$ modulo $9$, and that is because this is the only class modulo $9$ that has only two representatives in $\D = \{ 0,1,\ldots,25\}$. %There is a total of $26$ words in $L_\alpha^3$. 

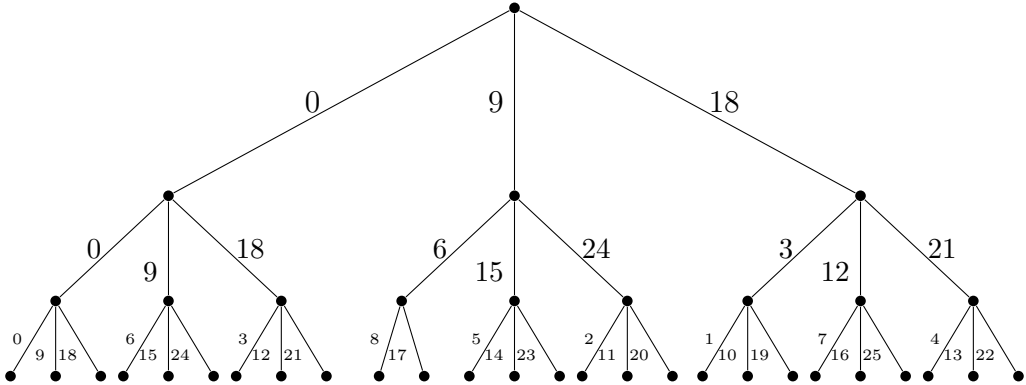
\begin{figure} 

\begin{tikzpicture} %[level distance=1.5cm, sibling distance=2.5cm]
        [level distance=2cm,
   level 1/.style={sibling distance=4.6cm, level distance = 2.5cm},
   level 2/.style={sibling distance=1.5cm, level distance = 1.4cm},
   level 3/.style={sibling distance=0.6cm, level distance = 1cm}]
    % Define styles for the nodes
    \tikzstyle{filled} = [circle, fill=black, minimum size=4pt, inner sep=0pt]
    
    % Root node
    \node[filled] {}
        child {node[filled] {}
                child {node[filled] {} 
                        child {node[filled] {} edge from parent node[left] {\tiny{0}}}
                        child {node[filled] {} edge from parent node[left, near end] {\tiny{9}}}
                        child {node[filled] {} edge from parent node[left, near end] {\tiny{18}}}
                    edge from parent node[left] {\small{0}}}
                child {node[filled] {}
                        child {node[filled] {} edge from parent node[left] {\tiny{6}}}
                        child {node[filled] {} edge from parent node[left, near end] {\tiny{15}}}
                        child {node[filled] {} edge from parent node[left, near end] {\tiny{24}}}
                    edge from parent node[left, near end] {\small{9}}}
                child {node[filled] {} 
                        child {node[filled] {} edge from parent node[left] {\tiny{3}}}
                        child {node[filled] {} edge from parent node[left, near end] {\tiny{12}}}
                        child {node[filled] {} edge from parent node[left, near end] {\tiny{21}}}
                    edge from parent node[right] {\small{18}}}
            edge from parent node[left] {0\,\,}}
        child {node[filled] {}             
                child {node[filled] {} 
                        child {node[filled] {} edge from parent node[left] {\tiny{8}}}
                        child {node[filled] {} edge from parent node[left, near end] {\tiny{17}}}
                    edge from parent node[left] {\small{6}}}
                child {node[filled] {} 
                        child {node[filled] {} edge from parent node[left] {\tiny{5}}}
                        child {node[filled] {} edge from parent node[left, near end] {\tiny{14}}}
                        child {node[filled] {} edge from parent node[left, near end] {\tiny{23}}}
                    edge from parent node[left, near end] {\small{15}}}
                child {node[filled] {} 
                        child {node[filled] {} edge from parent node[left] {\tiny{2}}}
                        child {node[filled] {} edge from parent node[left, near end] {\tiny{11}}}
                        child {node[filled] {} edge from parent node[left, near end] {\tiny{20}}}
                    edge from parent node[right] {\small{24}}}
            edge from parent node[left] {9}}
        child {node[filled] {} 
                child {node[filled] {} 
                        child {node[filled] {} edge from parent node[left] {\tiny{1}}}
                        child {node[filled] {} edge from parent node[left, near end] {\tiny{10}}}
                        child {node[filled] {} edge from parent node[left, near end] {\tiny{19}}}
                    edge from parent node[left] {\small{3}}}
                child {node[filled] {} 
                        child {node[filled] {} edge from parent node[left] {\tiny{7}}}
                        child {node[filled] {} edge from parent node[left, near end] {\tiny{16}}}
                        child {node[filled] {} edge from parent node[left, near end] {\tiny{25}}}
                    edge from parent node[left, near end] {\small{12}}}
                child {node[filled] {} 
                        child {node[filled] {} edge from parent node[left] {\tiny{4}}}
                        child {node[filled] {} edge from parent node[left, near end] {\tiny{13}}}
                        child {node[filled] {} edge from parent node[left, near end] {\tiny{22}}}
                    edge from parent node[right] {\small{21}}}
            edge from parent node[right] {\,\,18}};
\end{tikzpicture}
\caption{The tree $\mathcal{T}_\alpha$ for $\alpha = \frac{-1+5i}{3}$.}
\label{exampletree2}
\end{figure}

Theorem~\ref{theorem:residueclasses} allows us to create the tree from Figure~\ref{exampletreeresidues}: if we denote by $\overline{r}$ the set $\D \cap (a_2 \Z + r)$, we can group the edges coming down from one node in Figure~\ref{exampletree2} and create a new tree where the corresponding residue class $\overline{r}$ is represented in a node. The nodes in Figure~\ref{exampletreeresidues} are arranged increasingly from left to right, so they are not in the same order as they were in Figure~\ref{exampletree2}. It could be an interesting line of research to find more patterns in these trees.

%\hfill \break
\begin{figure}
\begin{tikzpicture} %[level distance=1.5cm, sibling distance=2.5cm]
        [level distance=2cm,
   level 1/.style={sibling distance=4.6cm, level distance = 2.5cm},
   level 2/.style={sibling distance=1.5cm, level distance = 1.4cm},
   level 3/.style={sibling distance=0.4cm, level distance = 1cm}]
    % Define styles for the nodes
    \tikzstyle{filled} = [circle, fill=black, minimum size=4pt, inner sep=0pt]
    
    % Root node
    \node { $\overline{0}$}
        child {node { $\overline{0}$}
                child {node { $\overline{0}$} 
                        child {node { $\overline{0}$} edge from parent node{~}}
                        child {node { $\overline{3}$} edge from parent node{~}}
                        child {node { $\overline{6}$} edge from parent node{~}}
                    edge from parent}
                child {node { $\overline{3}$}
                        child {node { $\overline{1}$} edge from parent node{~}}
                        child {node { $\overline{4}$} edge from parent node{~}}
                        child {node { $\overline{7}$} edge from parent node{~}}
                    edge from parent node{~}}
                child {node { $\overline{6}$} 
                        child {node { $\overline{2}$} edge from parent node{~}}
                        child {node { $\overline{5}$} edge from parent node{~}}
                        child {node { $\overline{8}$} edge from parent node{~}}
                    edge from parent node{~}}
            edge from parent node{~}}
        child {node { $\overline{3}$}             
                child {node { $\overline{1}$} 
                        child {node { $\overline{1}$} edge from parent node{~}}
                        child {node { $\overline{4}$} edge from parent node{~}}
                        child {node { $\overline{7}$} edge from parent node{~}}
                    edge from parent node{~}}
                child {node { $\overline{4}$} 
                        child {node { $\overline{2}$} edge from parent node{~}}
                        child {node { $\overline{5}$} edge from parent node{~}}
                        child {node { $\overline{8}$} edge from parent node{~}}
                    edge from parent node{~}}
                child {node { $\overline{7}$} 
                        child {node { $\overline{0}$} edge from parent node{~}}
                        child {node { $\overline{3}$} edge from parent node{~}}
                        child {node { $\overline{6}$} edge from parent node{~}}
                    edge from parent node{~}}
            edge from parent node{~}}
        child {node { $\overline{6}$} 
                child {node { $\overline{2}$} 
                        child {node { $\overline{1}$} edge from parent node{~}}
                        child {node { $\overline{4}$} edge from parent node{~}}
                        child {node { $\overline{7}$} edge from parent node{~}}
                    edge from parent node{~}}
                child {node { $\overline{5}$} 
                        child {node { $\overline{2}$} edge from parent node{~}}
                        child {node { $\overline{5}$} edge from parent node{~}}
                        child {node { $\overline{8}$} edge from parent node{~}}
                    edge from parent node{~}}
                child {node { $\overline{8}$} 
                        child {node { $\overline{3}$} edge from parent node{~}}
                        child {node { $\overline{6}$} edge from parent node{~}}
                    edge from parent node{~}}
            edge from parent node{~}};
\end{tikzpicture}
\caption{The tree $\mathcal{T}_\alpha$ for $\alpha = \frac{-1+5i}{3}$ but naming each node by the residue class of its children.}
\label{exampletreeresidues} 
\end{figure}
%\hfill \break

Next, we show that the length of an integer $\alpha$-expansion grows logarithmically with respect to the modulus of $N$. As usual, given a complex number $a+bi$ we denote $|a+bi|=\sqrt{a^2+b^2}$.

\begin{proposition}

If $N \in \Lambda_\alpha$ has an integer $\alpha$-expansion  
\[
N =  \sum_{j=0}^{k} d_j \alpha^j,
\]
then it holds that
\[
    k = \log_{|\alpha|}(|N|)+ \mathcal{O}(1).
\]

\end{proposition}
\begin{proof}
    % In Morgenbesser steiner Thus CITE lemma 4.1 they show that, for positive rational bases $\ab > 0$. For negative $\ab$ it is similar. We have
    Using geometric sums and the fact that \( 0 \leq d_j \leq |a_0|-1 \) we get the upper bound
\[
|N| = \left| \sum_{j=0}^{k} d_j \alpha^j \right| \leq \sum_{j=0}^{k} |d_j| |\alpha|^j \leq (|a_0|-1) \sum_{j=0}^{k} |\alpha|^j = (|a_0|-1) \frac{|\alpha|^{k+1}-1}{|\alpha|-1}.\]

For the lower bound,
\[
|N| \geq |d_k| |\alpha|^k - \sum_{j=0}^{k-1} |d_j| |\alpha|^j.
\]
Since \( 1 \leq d_k \leq |a_0|-1 \) and \( 0 \leq d_j \leq |a_0|-1 \) for \( j < k \) we get:
\[
|N| \geq |\alpha|^k - (|a_0|-1) \sum_{j=0}^{k-1} |\alpha|^j = |\alpha|^k - (|a_0|-1) \frac{|\alpha|^k - 1}{|\alpha|-1} = \left( 1 - \frac{|a_0|-1}{|\alpha|-1} \right) |\alpha|^k + \frac{|a_0|-1}{|\alpha|-1}.
\]
Therefore
    \[
\left( 1 - \frac{|a_0|-1}{|\alpha|-1} \right) |\alpha|^k + \frac{|a_0|-1}{|\alpha|-1} \leq |N| \leq (|a_0|-1) \frac{|\alpha|^{k+1}-1}{|\alpha|-1}.
\]
Taking the logarithm in base $|\alpha|$ on the inequality yields the desired result.
\end{proof}

% \subsection{Distribution of digits}

% The digits are uniformly distributed

\subsection{Algorithms for addition and multiplication}\label{Algorithms for addition and multiplication}

Given a number system $(\alpha, \D)$ with the finiteness property in $\Lambda_\alpha$ and the expansions of two points $N, M  \in \Lambda_\alpha$, we can compute the integer $\alpha$-expansion of $N+M$. When adding digit-wise, starting from the right, if the sum of the digits exceeds $|a_0|-1$ we need to subtract $|a_0|$ and carry the integer $\alpha$-expansion of $\frac{|a_0|}{\alpha}$ to the next digits to the left. It holds that $\frac{|a_0|}{\alpha} \in \Lambda_\alpha$: we have $a_0 = -a_2\alpha^2-a_1\alpha$ and therefore $\frac{|a_0|}{\alpha} = |a_2 \alpha + a_1| \in \Lambda_\alpha$. The algorithm ends after all digits have been added and there is no carry, or after a sufficiently large trail of zeroes appears to the left. One can analogously compute the expansion of $N \cdot M$.

% \begin{algorithm}
% Let  $N = (d^{(N)}_{k(N)} \ldots d^{(N)}_0)_\alpha$ and $M = (d^{(M)}_{k(M)} \ldots d^{(M)}_0)_\alpha$.
% \begin{enumerate}
%     \item We first add digit-wise $d_j = d^{(N)}_j + d^{(M)}_j$ for $0 \leq j \leq \max\{k(N), k(M)\}$ by padding the shortest expansions with zeroes to the left.
%     \item Compute the integer $\alpha$-expansion $\frac{|a_0|}{\alpha} = (e_l \ldots e_0 )_\alpha$.
%     \item Set $j =0$.
%     \item Find $n \in \N$ so that $d_j - n|a_0| \in \D$ and replace $d_j \leftarrow d_j - n|a_0|$.
%     \item Replace $d_{j+i} \leftarrow d_{j+i} + ne_{i}$ for $0 \geq i \geq l$.
%     \item Set $j \leftarrow j+1$ and repeat steps $(4)$ and $(5)$.
%     \item Break after all digits have been added and there is no carry, or after there is a trail of $L$ zeroes to the left for sufficiently large $L$ (it would suffice $L > 2(k(N) + k(M))$. 
%     \item Remove all the zeroes to the left, so $N+M = (d_k \ldots d_0)_\alpha$.
% \end{enumerate}
% \end{algorithm}

% Now that we know how to add, it is not hard to multiply in base $\alpha$: one simply needs to compute digit-wise multiplication, then shift and add accordingly, and then reduce each digit modulo $|a_0|$ and carry to the following digits.

\begin{example}
    Let $\alpha = \frac{-1+3i}{2}$ as before. We have $\frac{|a_0|}{\alpha} = -1-3i = (203)_\alpha$. For the sum, consider $-8 = (442)_\alpha$ and $5-3i = (2234)_\alpha$. We see in \eqref{sumandproduct} (left) that $-3-3i = (201)_\alpha$. For the multiplication, consider $-2 = (223)_\alpha$ and $6i = (42)_\alpha$. It is shown in \eqref{sumandproduct} (right) that $-12i = (2232141)_\alpha$. The auxiliary computations are depicted in gray, and the carriers are at the bottom.
    \begin{equation}\label{sumandproduct}
        \begin{array}[t]{*{7}{r}}
                & & &  4 & 4 & 2 \\
              &+ & 2 & 2 & 3 & 4 \\ \hline
            \color{gray} \ldots \; \color{gray} 10 &  \color{gray} 10 & \color{gray} 10 & \color{gray} 12  & \color{gray} 10 &  \color{gray}6 \\
            \ldots \; 0 & 0 & 0 & 2 & 0 & 1 \\ \hline
            & &  \color{gray} 2 & \color{gray}0 & \color{gray}3 & \\
            &  \color{gray} 2 & \color{gray}0 & \color{gray}3 & &\\
            &  \color{gray} 2 & \color{gray}0 & \color{gray}3 &  & \\
            \color{gray} 2 & \color{gray}0 & \color{gray}3 & &  \\
             \color{gray} 2 & \color{gray}0 & \color{gray}3 & &  \\
             & \color{gray} \vdots &
        \end{array}
            \,\, \,\,\,\, \;\;\;\;\;\;\;
        \begin{array}[t]{*{7}{r}}              
             && & & 2 & 2 & 3  \\ 
              & & & & \times  & 4 & 2 \\ \hline
               & &  &  & \color{gray} 4  & \color{gray} 4 &  \color{gray}6 \\
               & & \color{gray} + & \color{gray} 8 & \color{gray} 8  & \color{gray} 12 &  \color{gray}0 \\ \hline
            & &  & \color{gray} 8 & \color{gray} 12  & \color{gray} 16 &  \color{gray}6 \\
            2 & 2 & 3 & 2 & 1 & 4 & 1 \\ \hline
           & & &  \color{gray} 2 & \color{gray}0 & \color{gray}3 & \\
           & &  \color{gray} 2 & \color{gray}0 & \color{gray}3 & &\\
           & &  \color{gray} 2 & \color{gray}0 & \color{gray}3 &  & \\
           & &  \color{gray} 2 & \color{gray}0 & \color{gray}3 & &\\
           & \color{gray} 2 & \color{gray}0 & \color{gray}3 & &  \\
           &  \color{gray} 2 & \color{gray}0 & \color{gray}3 & &  \\
             & \color{gray} \vdots &
        \end{array}
    \end{equation}

\end{example}

\section{Expansions of complex numbers}\label{Expansion of complex numbers}

Our next step in the study of algebraic number systems is to consider expansions of complex numbers, by introducing a radix point, or in other words, by allowing digits $d_j$ where $j<0$.

    \begin{definition}
    An \textbf{$\alpha$-expansion} of a complex number $x$ is an expansion of the form

    \begin{equation}\label{eq:alphaexpansion}
        x = \sum_{j \leq k} d_j \alpha^j \qquad (d_j\in\mathcal{D})
    \end{equation}
    
    with $d_k \neq 0$, such that $(d_k \ldots d_l)_{\alpha}$ corresponds to the integer $\alpha$-expansion of some $N \in \Lambda_\alpha$ for every $l \leq k$. We denote it
    \[
        x = (d_k \ldots d_0 . d_{-1}d_{-2}\ldots)_{\alpha}.
    \]
    It follows from the definition of $\mathcal{T}_\alpha$ that \eqref{eq:alphaexpansion} is an $\alpha$-expansion of $x$ if and only if the sequence $(d_j)_{j \leq k}$ corresponds to an infinite downward path in the tree $\mathcal{T}_\alpha$ starting from the root.
    \end{definition}

\begin{remark}
    If we add a trail of zeroes to the right of the integer $\alpha$-expansion of $N \in \Lambda_\alpha$, in general, we do not get an $\alpha$-expansion of $N$: take, for example, $\alpha = \frac{-1+3i}{2}$. Then $2 = (2)_\alpha$, however $(2.0000\ldots)_\alpha$ is not an $\alpha$-expansion of $2$, because the path $200$ is not in the tree $\mathcal{T}_\alpha$. %Integer $\alpha$-expansions and $\alpha$-expansions are two different things because the first is always finite, and the second is always infinite. %, and they do not coincide in $\Lambda_\alpha$. %In the upcoming sections, we will see that integer $\alpha$-expansions of points in $\Lambda_\alpha$ and $\alpha$-expansions of complex numbers can be seen as particular cases of $\alpha$-expansions of what we will call ambinumbers, which will arise as a natural set for expansions in algebraic bases.
\end{remark}

Next, we give an algorithm to approximate an $\alpha$-expansion of a given $x \in \C$. As a first step, we introduce a function $\lambda_\alpha$ that maps $x$ to a nearby point in $\Lambda_\alpha$.

Regard $\C$ as a two dimensional $\R$-vector space with basis $\{1,i\}$. Let $\mathcal{B}_\alpha = \{ a_2, a_2\alpha + a_1 \}$ be the Brunotte basis of $\Lambda_\alpha$. Consider the linear isomorphism
\begin{equation}
    f: \C \rightarrow \C, \qquad 
    \begin{cases}
        1 \mapsto a_2,\\
        i \mapsto a_2 \alpha + a_1.
    \end{cases}
\end{equation}
Consider the complex floor function defined as $\lfloor a + bi \rfloor := \lfloor a \rfloor + \lfloor b \rfloor i$. Define
\[
    \lambda_\alpha: \C \rightarrow \Lambda_\alpha, \qquad x \mapsto f(\lfloor f^{-1}(x) \rfloor).
\]
%Given $x \in \C$, we give an algorithm that converges to an $\alpha$-expansion of $x$. 

\begin{algorithm}\label{algorithmcomplexexpansion}
For $n \in \N$:
\begin{enumerate}
    \item Compute $\lambda_\alpha( \alpha^n x) \in \Lambda_\alpha.$
    \item Find its integer $\alpha$-expansion $\lambda_\alpha( \alpha^n x) = (d_k \ldots d_0)_\alpha .$
    \item Let $\textbf{w}_n := d_k \ldots d_{k-n} . d_{k-n-1} \ldots d_0 00 \ldots \in \D^{\omega}$.% with additional zeroes to the left if necessary. %Place the radix point on $\textbf{w}_n$ before $n$ decimals.
\end{enumerate}

\end{algorithm}

Algorithm \ref{algorithmcomplexexpansion} produces a sequence of expansions in base $\alpha$. They are, in general, not $\alpha$-expansions, but we will see that they converge to an $\alpha$-expansion of $x$ in the following metric: given any $\textbf{x} = (x_j)_{j\in\N}$ and $\textbf{y} = (y_j)_{j\in\N}$ in $\D^{\omega}$ consider the distance $dist(\textbf{x}, \textbf{y}) = 2^{ -\inf \{ j \in \N \,:\, x_j \neq y_j \} }$.

\begin{proposition}\label{convergencealgorithm}
    The sequence $\textbf{w}_n$ from Algorithm \ref{algorithmcomplexexpansion} converges to an $\alpha$-expansion of~$x$.
\end{proposition}

\begin{proof}
    We first show that $x = \lim_{n \to \infty} \pi_\alpha(\textbf{w}_n) $, where $\pi_\alpha$ is the evaluation map in base $\alpha$. It is clear that $$\pi_\alpha(\textbf{w}_n) = \alpha^{-n} \pi_\alpha(d_k \ldots d_0) = \alpha^{-n} \lambda_\alpha( \alpha^n x).$$

    We have 
    $$\lambda_\alpha(\alpha^n x) = f(\lfloor f^{-1}(\alpha^n x) \rfloor) = f(f^{-1}(\alpha^n x) + \{f^{-1}(\alpha^n x)\}),$$ 
    where $\{ z \} = z - \lfloor z \rfloor $ is the complex fractional part. Since $f$ is a linear map, 
    $$ f(f^{-1}(\alpha^n x) + \{f^{-1}(\alpha^n x)\}) = \alpha^n x + f(\{f^{-1}(\alpha^n x)\}) .$$
    By definition of the complex floor, it holds that $| \{ z \} | < 2$ for any complex number $z$. Since $f$ is linear, it is bounded, and so $\varepsilon_n :=  f(\{f^{-1}(\alpha^n x) \} )$ is bounded. We get
    $$ \pi_\alpha(\textbf{w}_n) = \alpha^{-n}(\alpha^n x + \varepsilon_n) \xrightarrow[n \to \infty]{} x .$$
    
     %Since $\pi_\alpha(\textbf{w}_n)$ converges to $x$ and $\pi_\alpha$ is continuos (PROVE), 
     We show now that the sequence $\textbf{w}_n$ converges in $\D^{\omega}$ to some $\textbf{w}$. For that we need to show that $N_n =  \lambda_\alpha(\alpha^n x)$ and $N_{n+1} = \lambda_\alpha(\alpha^{n+1} x)$ have integer $\alpha$-expansions with a common prefix whose length goes to infinity as $n$ grows. We have

    \[
        T_\alpha(N_{n+1}) - N_n = \frac{N_{n+1}-d_0}{\alpha} - N_n =  \frac{ \varepsilon_{n+1}}{\alpha} - \varepsilon_n - \frac{d_0}{\alpha}.
    \]

    The integer $\alpha$-expansion of $T_\alpha(N_{n+1})$ is the integer $\alpha$-expansion of $N_{n+1}$ with the last digit $d_0$ removed. Let $\langle N_n \rangle_\alpha = d_k \ldots d_0$ and $\langle T_\alpha(N_{n+1})\rangle_\alpha = d'_k\ldots d_0'$ where we patch one of the expansions with zeroes to the left so that they have the same length. Let $l$ be the largest index so that $d_l \neq d'_l$. Then $|\alpha|^l \leq |T_\alpha(N_{n+1}) - N_n| = \left|  \frac{ \varepsilon_{n+1}}{\alpha} - \varepsilon_n - \frac{d_0}{\alpha} \right|$, which is bounded, and therefore $l$ is bounded. As a consequence, the distance in $\D^{\omega}$ of $\textbf{w}_n$ and $\textbf{w}_{n+1}$ tends to $0$ as $n$ goes to infinity. Therefore, the sequence is Cauchy and since $\D^{\omega}$ is complete, it converges to $\textbf{w} \in \D^{\omega}$. Since the prefixes of $\textbf{w}$ are integer $\alpha$-expansions by definition, $\textbf{w}$ is an $\alpha$-expansion of $x$.
\end{proof}

We illustrate this algorithm with our recurrent example.

\begin{example}\label{examplesqrt2}
    Let $\alpha = \frac{-1+3i}{2}$, then $\Lambda_\alpha$ is spanned by $\{ 2,1+ 3i\}$. The map $\lambda_\alpha: \C \rightarrow \Lambda_\alpha$ is given by
    \[
        \lambda_\alpha(x+yi) = 2 \left\lfloor \frac{x}{2} - \frac{y}{6} \right\rfloor + \left\lfloor \frac{y}{3} \right\rfloor + 3 \left\lfloor \frac{y}{3} \right\rfloor i.
    \]
    We want to find the $\alpha$-expansion of $\sqrt{2}$. Table \ref{approxsqrt2} shows $ \lambda_\alpha( \alpha^n \sqrt{2})$, its integer $\alpha$-expansion and $\textbf{w}_n$ for $n=1,2,\ldots,15$. The approximation for $n=50$ is $$2.23411214244400202412000344114424444410323402111430.$$ Applying $\pi_\alpha$ to this digit string gives as a result $$1.414213562226875 + 4.779186057674623 \cdot 10^{-11}i,$$ which is a good approximation of $\sqrt{2}$.
    
\end{example}

We mention that, in the definition of the map $\lambda_\alpha$, we could have used the ceiling function $\lceil \cdot \rceil$ in the real and/or imaginary parts instead of the floor. In all cases, the expansion obtained is the same.

   \begin{table}
        \centering
        \begin{tabular}{ccc}
        %\flushleft 
        \toprule
        $\lambda_\alpha( \alpha^n \sqrt{2})$ & $\langle \lambda_\alpha( \alpha^n \sqrt{2}) \rangle_\alpha$ & $\textbf{w}_n$\\
        \cmidrule(lr){1-3}
        -2 & 223 &
22.30000000000000\ldots \\
- 5-3i  & 424 &
4.240000000000000\ldots \\
 2-6i  & 2210 &
2.210000000000000\ldots \\
6i & 42 &
0.004200000000000\ldots \\
- 15 -3i & 201114 & 2.011140000000000\ldots \\
 7-21i  & 2234110 &
2.234110000000000\ldots \\
 25+21i & 22341322 &
2.234132200000000\ldots \\
- 49+27i  & 223413440 &
2.234134400000000\ldots \\
- 23 -87i & 2234112343 &
2.234112343000000\ldots \\
 137 + 9i  & 22341121400 &
2.234112140000000\ldots \\
- 85+201i  & 223411214222 &
2.234112142220000\ldots \\
- 260 -228i & 2234112142444 &
2.234112142444000\ldots \\
 470-276i  & 22341121422413 &
2.234112142241300\ldots \\
177 + 843i  & 223411214222103 &
2.234112142221030\ldots \\
- 1358 -156i  & 2234112142444000 &
2.234112142444000\ldots \\
\bottomrule
        \end{tabular}
        \caption{Approximating the $\alpha$-expansion of $\sqrt{2}$ for $\alpha = \frac{-1+3i}{2}$.}
        \label{approxsqrt2}
    \end{table}

\section{Tilings for algebraic number systems}\label{Tilings for alpha expansions}

We introduce next a family of sets that we denote $\mathcal{G}_\alpha(N)$, originally studied in~\citep{MR3391902}, where they are called intersective tiles. $\mathcal{G}_\alpha(N)$ is the set of points in $\C$ with an $\alpha$-expansion whose integer part is $N$. We will show later on that these sets form a tiling of the complex plane.

\begin{definition}
    Let $(\alpha, \D)$ be an algebraic number system and $N \in \Lambda_\alpha$. Let
    \[ 
        \mathcal{G}_\alpha(N) := \{ x \in \C\,:\, x = (d_k\ldots d_0.d_{-1}d_{-2}\ldots)_\alpha \mbox{ and } N=(d_k \ldots d_0)_\alpha \}.
    \]    
\end{definition}

% \begin{figure}
%     \centering
%     \includegraphics[width=0.5\linewidth]{label_tiles_lattice_n=2}
%     \caption{Tiles $\mathcal{G}_\alpha(N)$ for $\alpha = \frac{-1+3i}{2}$ and $N \in \Lambda_\alpha$}
%     \label{fig:label_tiles_lattice_n=2}
% \end{figure}

% \begin{figure}
%     \centering
%     \includegraphics[width=0.8\linewidth]{label_tiles_lattice_n=3}
%     \caption{Tiles $\mathcal{G}_\alpha(N)$ for $\alpha = \frac{-1+3i}{2}$ and $N \in \Lambda_\alpha$}
%     \label{fig:label_tiles_lattice_n=3}
% \end{figure}

In Figure~\ref{fig:label tiles lattice}, we show $\mathcal{G}_\alpha(N)$ for $\alpha = \frac{-1+3i}{2}$ for $N \in \Lambda_\alpha$ whose integer $\alpha$-expansion has length at most three. To each picture corresponds one level of the tree $\mathcal{T}_\alpha$ from Figure~\ref{exampletree}. The lattice points are labeled (the place of the label is where the corresponding point is), as well as the digits of their integer $\alpha$-expansion. We chose a four-coloring for this tiling, and one can observe that, even though all of the tiles seem to be different, the ones of the same color have a similar shape. We will explain this in the next section.

% An analogous illustration for $\alpha = \frac{-1+5i}{3}$ is shown in Figure~\ref{fig:label tiles lattice other example}, where we consider points whose integer $\alpha$-expansion has at most three digits.

\begin{figure}
    \centering
    \includegraphics[width=0.4\linewidth]{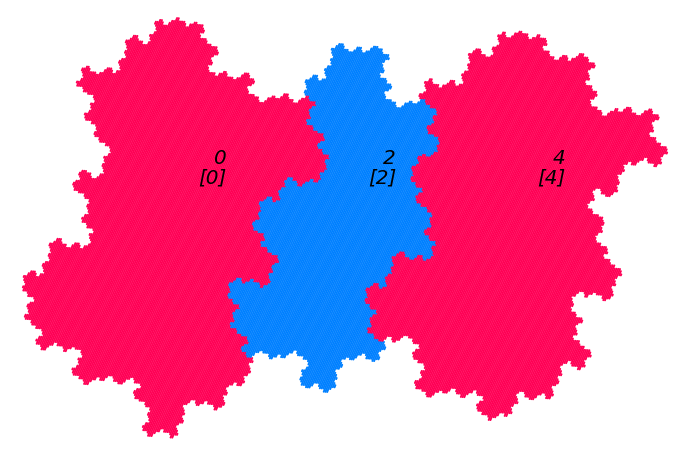}\,\includegraphics[width=0.4\linewidth]{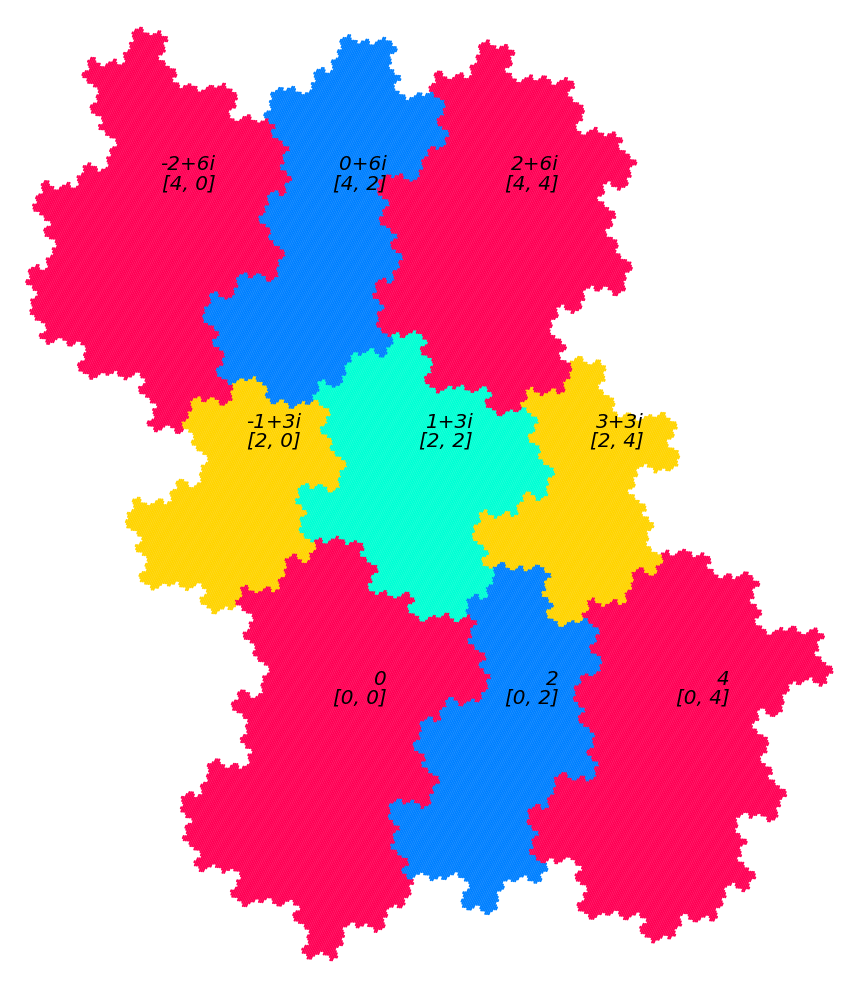}
    
    \includegraphics[width=0.7\linewidth]{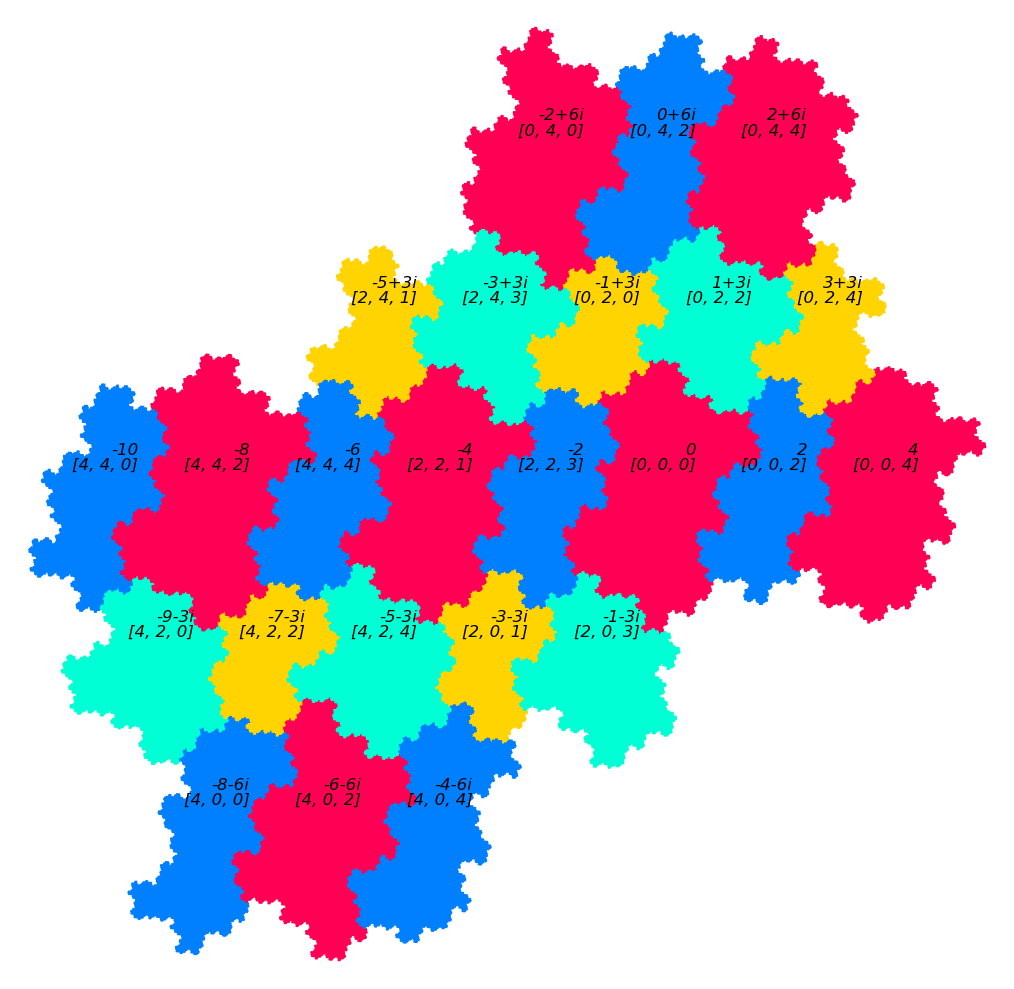}
    \caption{Tiles $\mathcal{G}_\alpha(N)$ for $\alpha = \frac{-1+3i}{2}$ and $N \in \Lambda_\alpha$}
    \label{fig:label tiles lattice}
\end{figure}

  \begin{definition}
        A collection $\mathcal{C}$ of compact subsets of $\C$ is called a \textbf{tiling} if it is a covering of $\C$ that is pairwise essentially disjoint, that is, the intersection of any two distinct sets has Lebesgue measure zero.
    \end{definition}

    \begin{proposition}\label{proposition:tiling}
        Let $(\alpha, \D)$ be an algebraic number system. Then the boundary of $ \mathcal{G}_\alpha(N) $ has Lebesgue measure zero for every $N \in \Lambda_\alpha$ and the collection $\mathcal{C} = \{ \mathcal{G}_\alpha(N) \,:\, N \in \Lambda_\alpha \}$ forms a tiling of $\C$.        
    \end{proposition}

    \begin{proof}
    This result can be found in \cite[Theorem 3]{MR3391902}. The authors define $\mathcal{G}_\alpha(N)$ in a different and more general way than us, by considering self-affine tiles in certain spaces defined in terms of $\mathfrak{p}$-adic completions of $\Q(\alpha)$ with respect to prime ideals. The equivalence between the two definitions is a consequence of Theorem~\ref{padiccompletions} from the next section.
    \end{proof}

    The intersection of two distinct sets in $\mathcal{C}$ occurs on their boundary because they are compact. We have arrived at our main result.
    
    \begin{theorem}\label{theorem:uniqueae}
        The $\alpha$-expansion of a complex number is unique almost everywhere.
    \end{theorem}

    \begin{proof}
        Let $x \in \C$ and suppose it has two different $\alpha$-expansions
\[x=\sum_{j \leq k} \alpha^j d_j = \sum_{j \leq k} \alpha^j d'_j,
\]
where $d_k\neq0$ and where we pad the second expansion with zeros if necessary. Let $m\leqslant k$ be the largest integer such that $d_{m}\neq d'_{m}$. Let $N:=(d_{k}\ldots d_{m+1} d_{m})_{\alpha}$ and $N':=(d'_{k}\ldots d'_{m+1}d'_m)_{\alpha}$. By definition of $\alpha$-expansion, we have $N, N' \in \Lambda_\alpha$ and they are distinct. It holds that $\alpha^{-m}x \in \mathcal{G}_\alpha(N) \cap \mathcal{G}_\alpha(N').$
As tiles only overlap on their boundaries, it holds in particular that $x \in \alpha^m \partial \mathcal{G}_\alpha(N) $. 
Therefore, a point $x \in \C$ has two different expansions if and only if $x \in \bigcup_{N \in \Lambda_\alpha}\bigcup_{m\in\Z} \alpha ^{m}\partial\mathcal{G}_\alpha(N) $.
Then $x$ is in a countable union of measure-zero sets.
    \end{proof}

\section{$\alpha$-expansions and $p$-adic completions}\label{p adic completions}
We now characterize $\alpha$-expansions in terms of $p$-adic completions of $\Q(i)$, where $p$ is a Gaussian prime. The ring of integers of the field of Gaussian rationals $\Q(i)$ is the ring of Gaussian integers $\Z[i].$ There are exactly four units of $\Z[i]$, namely $\{ \pm 1, \pm i \}.$

\begin{definition}

  A Gaussian integer $p \in \Z[i]$ is a \textbf{Gaussian prime} if $p \neq 0$, $p$ is not a unit and whenever $p$ divides a product $ab$ with $a,b \in \Z[i]$ then $p$ divides $a$ or $p$ divides $b$.
     Given a Gaussian prime $p$ and a unit $u$, the number $up$ is a Gaussian prime \textit{associated} to $p$. 
      
\end{definition}

Every Gaussian integer, and therefore every Gaussian rational, has a unique factorization into Gaussian primes. This is a classical result.

% \begin{figure}
%     \centering
%     \includegraphics[width=0.7\linewidth]{label tiles lattice other example}
%     \caption{Tiles $\mathcal{G}_\alpha(N)$ for $\alpha = \frac{-1+5i}{3}$ and $N \in \Lambda_\alpha$}
%     \label{fig:label tiles lattice other example}
% \end{figure}

 \begin{proposition}\label{primedecomp}
Given a non-zero $x \in \Q(i)$, it can be expressed as a product
\begin{equation}
x = u \, p_1^{e_1} \cdots p_k^{e_k}
\end{equation}

where $u \in \{\pm 1, \pm i \}$, $e_1,\ldots, e_k \in \Z$, and $p_1, \ldots, p_k$ are pairwise not associated. This factorization is unique up to associated primes and multiplication by units. Note that the exponents $e_j$ might be negative because $\Q(i)$ is the field of fractions of $\Z[i]$. 
\end{proposition}

%Note that we can always choose $u \in \{\pm 1, \pm i \}$ such that $\mbox{Re}(up) > 0$ and $\mbox{Im}(up) \geq 0$.

The following generalizes the notion of $p$-adic valuation in $\Q$.

\begin{definition}%[$p$-adic valuation and absolute value]
    Let $p$ be a Gaussian prime. We define the \textbf{$p$-adic valuation} as 
    $$\nu_p : \Q(i) \rightarrow \Z, \qquad x \mapsto \nu_p(x),$$
    where $\nu_p(x)$ is given by the exponent of $p$ in the decomposition of $x$ into Gaussian primes given in Proposition~\ref{primedecomp} whenever $x \neq 0$, and $\nu_p(0) = \infty$. We define the \textbf{$p$-adic absolute value}
\[ |\cdot|_p : \Q(i) \rightarrow \R, \qquad x \mapsto N(p)^{-\nu_p(x)}, \]
where $N(p)$ is the complex norm of $p$, with the convention that $|0|_p = N(p)^{-\infty} = 0.$ 
\end{definition}

Given a Gaussian rational $\alpha$, find two Gaussian integers $num(\alpha)$ and $den(\alpha)$ that are coprime (that is, they have no common Gaussian prime divisors), such that $$\alpha = \frac{num(\alpha)}{den(\alpha)}.$$ It follows from~\cite[Lemma 3.3]{MR3391902} that $N(num(\alpha))=|a_0|$ and $N(den(\alpha))=a_2$.

The proof of the next result follows the proof of~\cite[Lemma 6.2]{MR3391902}.

\begin{theorem}\label{padiccompletions}
      An expansion
        \begin{equation}\label{alphaexpansion}
             x = \sum_{j \leq k} d_j \alpha^j \qquad (d_j\in\mathcal{D})
        \end{equation}
    is an $\alpha$-expansion of $x$ if and only if it converges to $0$ in $K_p$ for every Gaussian prime $p$ dividing $den(\alpha)$.
\end{theorem}

    \begin{proof}
        Given $l \leq k$, consider the sum of the first $k-l$ terms of the series \eqref{alphaexpansion},
        \[
           S_l := d_k\alpha^{k} + \cdots + d_{l+1}\alpha^{l+1} + d_{l}\alpha^{l}.
        \]
    
        Let $N_l =\alpha^{-l} S_l$. By definition, the series is an $\alpha$-expansion if and only if $N_l \in \Lambda_\alpha$ for all $l \leq k$.        
        
         Assume that \eqref{alphaexpansion} is an $\alpha$-expansion of $x \in \C$. Then, by definition of $\Lambda_\alpha$, $N_l \in \Z[\alpha] \cap \alpha^{-1}\Z[\alpha^{-1}]$ and hence $S_l \in \alpha^{l-1} \Z[\alpha^{-1}]$. Let $p$ be a Gaussian prime dividing $den(\alpha)$, then $p^{-l+1}$ is a factor of the numerator of $\alpha^{l-1}$, and $p$ does not appear in the denominator of any point of $\Z[\alpha^{-1}]$, therefore $\nu_p(S_l) \geq -l + 1$. Therefore, $\lim_{l \to -\infty} \nu_p(S_l) = \infty$ and consequently $| \sum_{j \leq k} d_j \alpha^j|_p = 0$.

        For the converse, suppose that \eqref{alphaexpansion} is not an $\alpha$-expansion, that is, there exists some $l \leq k$ such that $N_l \notin \Lambda_\alpha$. Since clearly $N_l \in \Z[\alpha]$, it holds that $N_l \notin \alpha^{-1} \Z[\alpha^{-1}]$, and hence $S_l \notin \alpha^{l-1}\Z[\alpha^{-1}]$. Note that $d_j \alpha^j\in \alpha^{l-1} \Z[\alpha^{-1}]$ for every $j \leq l-1$. %Hence, 
        % $$\sum_{j \leq l-1} \alpha^j d_j \in \alpha^{l-1} \Z[\alpha^{-1}].$$
        % for every $t <l$, $$S_l + \sum_{j = k}^{t} \alpha^j d_j \notin \alpha^{l-1} \Z[\alpha^{-1}].$$

        Let $\varphi_p: \Q(i) \rightarrow K_p$ be the canonical embedding of $\Q(i)$ into its completion $K_p$. Then
        \[
            \sum_{j \leq k} \varphi_p(d_j \alpha^j ) \notin \varphi_p(\alpha^{l-1} \Z[\alpha^{-1}])
        \]
        and in particular
        \[
            \sum_{j \leq k} \varphi_p(d_j \alpha^j ) \neq 0.
        \]
    \end{proof}

We illustrate this result with our recurring example.

\begin{example}
     Let $\alpha = \frac{-1+3i}{2}$, root of $P_\alpha(X) = 2X^2 + 2X + 5$, and $\D = \{0,1,2,3,4\}$. We can express $\alpha$ as a quotient of coprime Gaussian integers as
        \[  \alpha = \frac{-2+i}{1+i} .\]
       
    Consider the Gaussian prime $p = 1+i$. We can define the $(1+i)$-adic valuation $\nu_{1+i}: \Q(i) \rightarrow \Z$. For example, $2$ factors as $2 = -i (1+i)^2,$ so $\nu_{1+i}(2) =2$ and hence $|2|_{1+i} = \frac14$. Let $K_{1+i}$ be the completion of $\Q(i)$ with respect to this absolute value. We can express an $\alpha$-expansion as
    \begin{equation}\label{fractionalseries}
        \sum_{j \leq k } d_j \alpha^{j}  =  \sum_{j = -k}^{\infty} \frac{d_{-j}}{(-2+i)^j} (1+i)^{j}.
    \end{equation}

    Following the proof of Theorem~\ref{padiccompletions}, the series~\eqref{fractionalseries} has to converge to $0$ in $K_{1+i}$; hence, we require that $1+i$ divides $d_k$ and therefore $d_k \in \{ 0,2,4\}.$   
   
\end{example}

     We now explain how the colors from Figure~\ref{fig:label tiles lattice} were chosen. We can choose $\{0,1\}$ as a set of residues for $\Z[i]/(1+i)\Z[i]$ and regard the elements $y \in K_{1+i}$ as series
    \begin{equation}\label{eq:gaussianpadic}
        y = \sum_{j=\nu_{1+i}(x)}^\infty (1+i)^j c_j, \qquad c_j \in \{0,1\},
    \end{equation}
    where $\nu_{1+i}$ extends the valuation to all of $K_{1+i}$ and satisfies $c_{\nu_{1+i}(y)} \neq 0$. Each $N \in \Lambda_\alpha$ can be expressed as $N = \sum_{j=1}^{\infty} (1+i)^j c_j$ because points in $\Lambda_\alpha$ are divisible by $1+i$. In Figure~\ref{fig:label tiles lattice}, we paint each tile $\mathcal{G}_\alpha (N)$ red whenever $c_1=0, c_2 = 0$, blue whenever $c_1=0, c_2 = 1$, yellow whenever $c_1=1,c_2 = 0$ and green whenever $c_1=1,c_2 = 1$. The next observation is that tiles of the same color seem to have a similar shape, that is, $\mathcal{G}_\alpha (N)$ and $\mathcal{G}_\alpha (N')$ have a more similar shape whenever $N$ and $N'$ are close in the $(1+i)$-adic distance. Note that $x \in \mathcal{G}_\alpha (N)$ if and only if $x = N + \sum_{j \leq -1} \alpha^j d_j$ where $\sum_{j \leq -1} \alpha^j d_j$ converges to $-N$ in $K_{1+i}$. One could therefore consider all series $\sum_{j \leq -1} \alpha^j d_j$ and take the limit both in $\C$ and in $K_{1+i}$, producing a set in $\C \times K_{1+i}$, and then $\mathcal{G}_\alpha (N)$ and $\mathcal{G}_\alpha (N')$ could be seen as `slices' of this set, and therefore they look more alike when $N$ and $N'$ are closer in $K_{1+i}$. This is depicted in Figure~\ref{24_slices_rendered_1}. This is the original formulation, which we formalize in the next section.

  \begin{figure}
      \centering
      \includegraphics[width=0.5\linewidth]{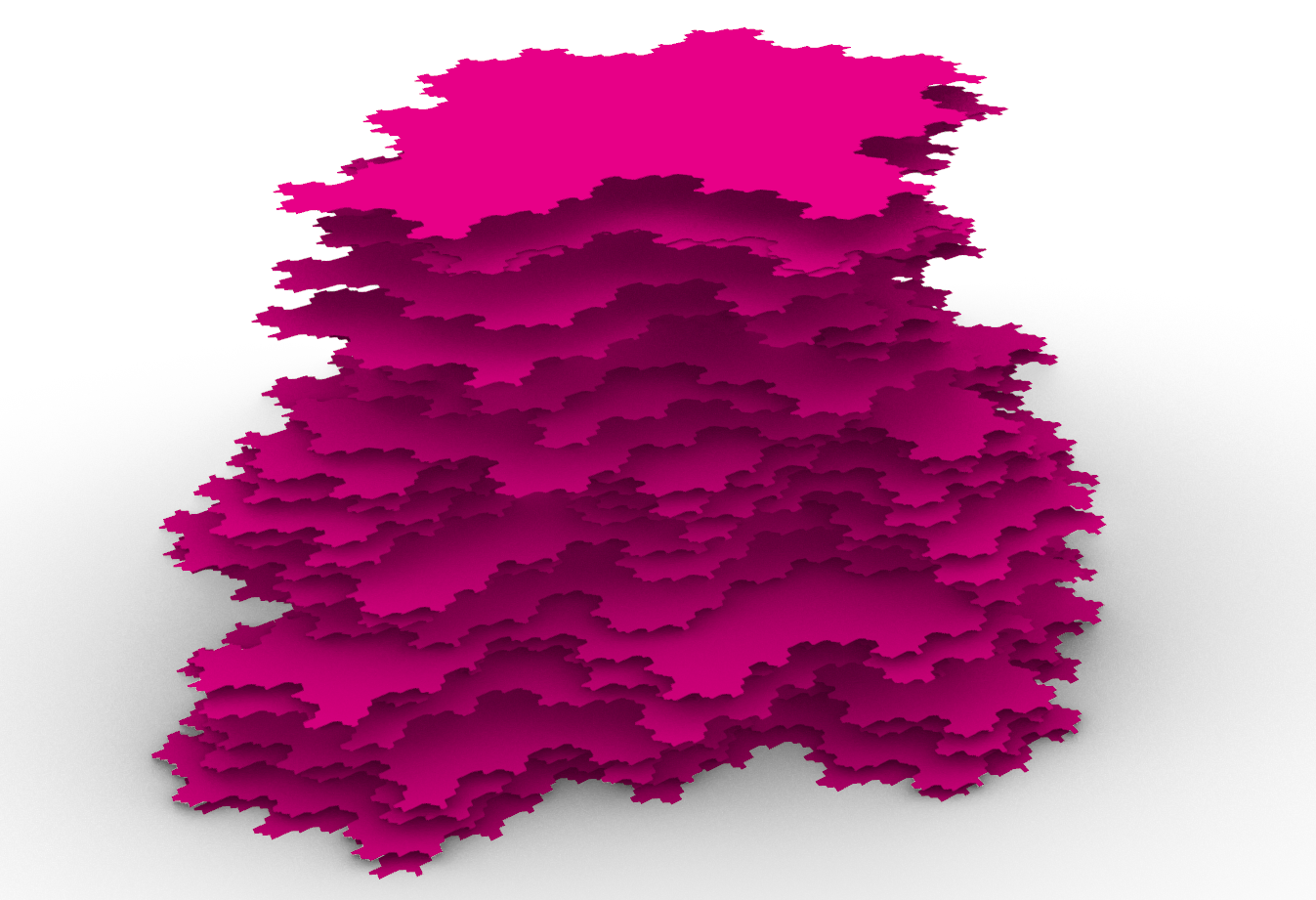}
      \caption{The set of series  $\sum_{j \leq -1} \alpha^j d_j$ embedded in $\C \times K_{1+i}$.}
      \label{24_slices_rendered_1}
  \end{figure}

\subsection{Ambinumbers}\label{Ambinumbers}

Given $x \in \C$, there are in general many ways to expand it in base $\alpha$ and digits $\D = \{0,\ldots,|a_0|-1\}$. We chose a particular kind of expansion and showed that it converges to $0$ in $K_p$ for all Gaussian primes $p$ dividing $den(\alpha)$. It seems, therefore, reasonable to look for an expansion that converges to $x$ in $\C$ and to some fixed value $y_p$ in each $K_p$. This is the setting introduced in~\citep{MR3391902}.

The article \citep{RT:22} considers a number system with base $-\frac32$ and digits $\D = \{0,1,2\}$ using the representation space $\R \times \Q_2$, and it is shown that each point in this space has an expansion in this number system that is unique almost everywhere with respect to a natural Haar measure; Donald Knuth called the elements of this space \textit{ambinumbers} in~\citep{Knuth22}. We retake this name here and generalize this notion for algebraic number systems.

\begin{definition}

Let $\alpha = \frac{num(\alpha)}{den(\alpha)}$. Define
\[
        K_{den(\alpha)} := \prod_{p |den(\alpha)} K_p
\]
where $p$ ranges over Gaussian primes. We call a pair $(x,y) \in \C \times  K_{den(\alpha)}$ an ambinumber. 
\end{definition}

 This motivates the following definition.

 \begin{definition}
    Consider a base $\alpha$ and digits $\D = \{0,\ldots,|a_0|-1\}$. Given an ambinumber $(x,y) \in \C \times K_{den(\alpha)}$, where $y = (y_p)_{p |den(\alpha)}$, the series
        \begin{equation}\label{ambinumberexpansion}
            \sum_{j \leq k} d_j \alpha^j \qquad (d_j\in\mathcal{D})
        \end{equation}
    $d_k \neq 0$ whenever $k\geq1$, is called an $\alpha$-expansion of $(x,y)$ if the series converges to $x$ in $\C$ and to $y_p$ in $K_p$ for every $p |den(\alpha)$. We note it
    \[
        (x,y) = (d_k \ldots d_0 . d_{-1}d_{-2} \ldots)_\alpha
    \]
 \end{definition}

We show next that $\alpha$-expansions of ambinumbers can take any form, and there are no restrictions on the strings of digits allowed. This justifies the introduction of $\C \times K_{den(\alpha)}$ as a natural representation space.

\begin{proposition}
    Let $(\alpha, \D)$ be an algebraic number system. A series of the form \eqref{ambinumberexpansion} 
    correspond to the $\alpha$-expansion of some ambinumber $(x,y) \in \C \times K_{den(\alpha)}$ for any sequence $(d_j)_{j \leq k}$.
\end{proposition}

\begin{proof}
The series converges in $\C$ because $|\alpha| > 1$ and $\D$ is finite. Let $p$ be a Gaussian prime dividing $den(\alpha)$, and let $r_p$ be the exponent of $p$ in the prime factorization of $den(\alpha)$. Let $l \leq k$ and consider the partial sum $S_l = \sum_{j=l}^{k} d_j \alpha^j$. Then $\nu_p(S_{l}-S_{l+1}) = \nu_p( d_l \alpha^l) = \nu_p(d_l)-l \, r_p $ and so $|S_l - S_{l+1}|_p = N(p)^{r_p l-\nu_p(d_l)}$ with $l \to -\infty$, hence $(S_l)$ is Cauchy in the complete space $K_p$, so the series converges in $K_p$. 
\end{proof}

Note that the $\alpha$-expansion of $x \in \C$ is then the $\alpha$-expansion of the ambinumber $(x,0)$. The integer $\alpha$-expansion of $N \in \Lambda_\alpha$ is the $\alpha$-expansion of the ambinumber $(N,N)$.

We now show a uniqueness result for $\alpha$-expansions of ambinumbers.

\begin{theorem}
    The $\alpha$-expansion of an ambinumber $(x, y) \in \C \times K_{den(\alpha)}$ is unique almost everywhere (with respect to a natural Haar measure).
\end{theorem}

\begin{proof}
    This result is a consequence of \cite[Theorem 2 and Corollary 1]{MR3391902}. The authors consider a set $\F$, which in our setting corresponds to all the ambinumbers with an expansion~\eqref{ambinumberexpansion} with $k \leq -1$, that is, the set of fractional expansions. They show that this set gives a tiling of $\C \times K_{den(\alpha)}$ (in their setting, this is done more generally). Given the existence of a tiling, the uniqueness almost everywhere can be obtained analogously as in the proof of Theorem~\ref{theorem:uniqueae}.
\end{proof}

%\subsection{Algorithms for expansions of ambinumbers}

To finalize, we show how to compute the expansion of an ambinumber $(x,y) \in \C \times K_{den(\alpha)}$ for the case $y \in \Lambda_\alpha$: we regard $\Lambda_\alpha$ as a subset of $K_{den(\alpha)}$ through the product of the canonical embeddings, that is, $y \hookrightarrow (y)_{p |den(\alpha)}$. Note that $\C \times K_{den(\alpha)}$ is a group with component-wise addition, therefore
\[
    (x,y) = (x,0) +(0,y).
\]
We know how to compute the $\alpha$-expansion of $(x,0)$, because it is exactly the $\alpha$-expansion of $x$. We also know how to add $\alpha$-expansions (same algorithm as with integer $\alpha$-expansions, just adding the radix point), so it remains to find the expansion of $(0,y)$. The group $\Lambda_\alpha$ acts on $\C \times K_{den(\alpha)}$ by multiplication, hence, if $y \in \Lambda_\alpha$, it holds that

\[
    (0,y) = y \cdot (0,1)
\]
Computing the $\alpha$-expansion of $(0,1)$ is not hard: we can first compute the $\alpha$-expansion of $-1$, which converges to $-1$ in $\R$ and to $0$ in $K_{den(\alpha)}$. Adding $1$ to the digit left to the radix point of this expansion will converge to $0$ in $\R$ and $1$ in $K_p$, yielding the expansion of $(0,1)$. Since $y \in \Lambda_\alpha$, we can compute its integer $\alpha$-expansion, and then multiply both expansions. 

\begin{example}
    As usual, let $\alpha = \frac{-1+3i}{2}$. Suppose we want to find the expansion of the ambinumber $(\sqrt{2},1+3i) \in \C \times K_{den(\alpha)}.$ From Example~\ref{examplesqrt2},
    \[
        (\sqrt{2}, 0) = (2.234112142444002024120003\ldots)_\alpha
    \]
    We also have
    \[
        -1 = (0.243100111243211314444112303\ldots)_\alpha
    \]
    Consequently, 
    \[
        (0,1) = (1.243100111243211314444112303\ldots)_\alpha
    \]
    Using $1+3i = (22)_\alpha$ and the multiplication algorithm, 
    \[
        (0,1+3i) = (0.12320244240042344032002044\ldots)_\alpha
    \]
    Using the addition algorithm, we get
    \[
        (\sqrt{2},1+3i) = (\sqrt{2}, 0) + (0,1+3i) = (2.13231231234442321444022 \ldots)_\alpha
    \]
    One can check that this expansion converges to the desired point. 
    % For the complex part, direct computation yields a good approximation of $\sqrt{2}$. For the $den(\alpha)$-adic part, given the expansion $(d_k \ldots d_0 . d_{-1}d_{-2} \ldots)_\alpha$ we first compute the partial sums $s_t = \sum_{j=-t}^k d_j \alpha^j$ for $t \leq k$. Then compute $\nu_den(\alpha)(s_t - y)$, which is the exponent of $den(\alpha)$ in the factorization of $s_t - y$. Check that $\nu_den(\alpha)(s_t - y) \xrightarrow[t\to \infty]{} \infty$, which implies that $\sum_{j=-\infty}^k d_j \alpha^j \xrightarrow[K_{den(\alpha)}]{} y$.
\end{example}

\section{Generalizations and open questions}\label{Open questions}

\subsection{Generalizations}

Throughout the paper, we assumed that $\alpha$ is a Gaussian rational, however, the study of algebraic number systems can also be done when $\alpha$ is an expanding quadratic algebraic number. The definition of the expansions is the same, and the main difference would be encountered when considering $ p$-adic completions. If $\alpha$ is not a Gaussian rational, then in general $\Q(\alpha)$ is not a unique factorization domain. However, it is a Dedekind domain, and the $p$-adic completions can be defined in terms of prime ideals. This is contained in~\citep{MR3391902} and therefore the proofs that rely on results there would stay essentially the same. Moreover, if $\alpha$ is an expanding algebraic number of degree $n$, instead of expansions in $\C$, one can consider expansions in $\R^n$ using Galois embeddings. We opted for $\alpha \in \Q(i)$ for simplicity. 

We considered digit sets of the form $\D = \{ 0, \ldots, |a_0|-1\}$, but it would suffice to consider a complete residue set for $\Z[\alpha] / \alpha \Z[\alpha]$ that is primitive, that is, such that $\Z \langle \alpha, \D \rangle = \Z[\alpha]$, where $\Z \langle \alpha, \D \rangle$ is the smallest $\alpha$-invariant $\Z$-submodule of $\Z[\alpha]$ containing the difference set $\D-\D$. Once again, for simplicity, we opted for this special shape of $\D$.

\subsection{Open questions}

We leave some open questions to motivate future research.

\begin{enumerate}

 \item Is the language of integer $\alpha$-expansions regular? Is it context-free? We conjecture it is not, as this is the case for rational base number systems (see Corollaries 7 and 9 of~\citep{MR2448050}).

 \item Is it possible to characterize, given a base $\alpha$, the set of points $x \in \C$ that have multiple $\alpha$-expansions? What is the upper bound for the number of possible $\alpha$-expansions? What shape do these multiple expansions have? This question is answered for the rational case; however, the algebraic setting is quite different because complex numbers are not ordered.

\item What can we say about the distribution of digits (and, respectively, finite words) in $\alpha$-expansions of complex numbers? We conjecture that, for almost all $x \in \C$, any two words of the same length in $D^*$ appear with the same frequency in the $\alpha$-expansion of $x$. This was proven for rational base number systems in~\citep{MST:13}. Not only is this not obvious, but it isn't easy to prove even in the rational case, and the authors in~\citep{MST:13} use sophisticated techniques such as Fourier analysis in Ad\`ele rings. 
 
\end{enumerate}

\bibliographystyle{elsarticle-harv}  
\bibliography{biblio2}

%%       \citet{<label>} ==> Lamport (1994)
%%       <label>} ==> (Lamport, 1994)
%%
%Example citation, See \citet{lamport94}.





% %% For authoryear reference style

% \bibitem[Lamport(1994)]{lamport94}
%   \textit{\LaTeX: a document preparation system},
%   Addison Wesley, Massachusetts,

% \end{thebibliography}
\end{document}